\documentclass[12pt]{article}
\usepackage{amsmath,amssymb,amsthm,bbm,srcltx,graphicx,a4wide,xr,xargs,natbib}
\usepackage[dvipsnames]{xcolor}
\usepackage[colorlinks]{hyperref}
\usepackage[notcite,notref]{showkeys}
\usepackage[shortlabels]{enumitem}
\numberwithin{equation}{section}

\newcommand{\wto}{\stackrel{w}{\longrightarrow}} %%%%%%%%%% weakly
\newcommand{\vto}{\stackrel{v}{\longrightarrow}} %%%%% vague 

\def\shift{B}

\newcommand\ind[1]{\mathbbm{1}{\left\{#1\right\}}}
\newcommandx{\norm}[2][1=]{\left|#2\right|_{#1}}

\newcommand{\EE}{\mathbb{E}}
\newcommand{\esp}{\mathbb{E}}

\newcommand{\Nset}{\mathbb{N}}
\newcommand{\PP}{\mathbb{P}} 
\newcommand{\pr}{\mathbb{P}} 
 
\newcommand{\Rset}{\mathbb{R}} 
 
\newcommand{\ZZ}{\mathbb{Z}} 
\newcommand{\Zset}{\mathbb{Z}} 

\newcommand{\setE}[1]{(\Rset^{#1})^\Zset}
\newcommand{\setEstar}[1]{\mathsf{E}^*_{#1}}
\newcommand{\setS}[1]{\mathsf{S}_{#1}}

\newcommand{\law}{\mathcal{L}}

\newcommand{\lpspace}[1]{\ell_{#1}}

\def\bzeta{\boldsymbol{\zeta}}

\def\spectralmeasure{\sigma}

\def\dtilde{\tilde{d}}
\def\proj{p}

\def\bszero{\boldsymbol{0}}

\newcommand{\bsx}{\boldsymbol{x}}
\newcommand{\bsy}{\boldsymbol{y}}
\newcommand{\bsX}{\boldsymbol{X}}
\newcommand{\bsY}{\boldsymbol{Y}}

\newcommand{\bsZ}{\boldsymbol{Z}}

\newcommand{\bstheta}{\boldsymbol{\theta}}
\newcommand{\bsTheta}{\boldsymbol{\Theta}}

\newcommand{\vbQ}{\boldsymbol{Q}}
\newcommand{\vbx}{\boldsymbol{x}}
\newcommand{\vby}{\boldsymbol{y}}
\newcommand{\vbX}{\boldsymbol{X}}
\newcommand{\vbY}{\boldsymbol{Y}}

\newcommand{\vbTheta}{\boldsymbol{\Theta}}

\def\PPC{N''}

 \def\mco{\mathcal{O}}

\def\mcm{\mathcal{M}}

\def\rmd{\mathrm{d}}
\def\rme{\mathrm{e}}

\newcommand{\lzero}{\ell_0}
\newcommand{\lo}{\tilde{\ell}_0}
\newcommand{\loo}{\tilde{\ell}_0\setminus\{\bszero\}}

\newcommandx{\sequ}[3][2=j,3=\mathbb{Z}]{\{#1_#2,#2\in#3\}}

\def\ie{i.e.}
\def\as{a.s.}
\def\wrt{w.r.t.}

\def\iid{i.i.d.}

\usepackage{cleveref}

\newtheorem{theorem}{Theorem}[section]
\newtheorem{lemma}[theorem]{Lemma}
\newtheorem{corollary}[theorem]{Corollary}
\newtheorem{proposition}[theorem]{Proposition}
\newtheorem{definition}[theorem]{Definition}

\theoremstyle{remark}
\newtheorem{remark}[theorem]{Remark}
\newtheorem{example}[theorem]{Example}

\crefname{hypothesis}{Assumption}{Assumptions}
\crefname{proposition}{Proposition}{Propositions}

\definecolor{mygray}{gray}{0.6}

\crefname{lemma}{Lemma}{Lemmas}

\begin{document}

\title{The tail process revisited} 

\author{Hrvoje Planini\'c\thanks{Department of Mathematics, University of Zagreb, Bijeni\v cka 30,
    Zagreb, Croatia}
     \and Philippe Soulier\thanks{Universit\'e Paris Nanterre, 92000 Nanterre, France}}

\maketitle

\begin{abstract}
  The tail measure of a regularly varying stationary time series has been recently introduced. It is
  used in this contribution to reconsider certain properties of the tail process and establish new
  ones. A new formulation of the time change formula is used to establish identities, some of which
  were indirectly known and some of which are new.
\end{abstract}

\section{Introduction}
Let $\sequ{\bsX}$ be a stationary regularly varying time series with values in $\Rset^d$. This means
that for all $s\leq t\in\Zset$ there exists a non-zero Radon measure $\nu_{s,t}$ on
$(\Rset^d)^{[s,t]}\setminus\{\bszero\}$ such that, as $u\to\infty$,
\begin{align}\label{eq:reg_var}
  \frac{\pr\left(u^{-1}\left(\bsX_s,\dots,\bsX_t\right) \in\cdot \right)}{\pr(\norm{\vbX_0}>u)}  \vto \nu_{s,t} \; , 
\end{align}
where $\vto$ denotes vague convergence.  Let $\norm\cdot$ denote an arbitrary norm on $\Rset^d$.
According to \cite[Theorem~2.1]{basrak:segers:2009}, this is equivalent to the existence of a
sequence $\sequ{\bsY}$ with $\pr(\norm{\vbY_0}>y)=y^{-\alpha}$ for $y\geq 1$ and such that for all
$s\leq t \in \Zset$, as $u\to\infty$,
\begin{align*}
  \law\left(u^{-1}\bsX_s,\dots,u^{-1}\bsX_t \mid
  \norm{\bsX_0}>u\right)  \wto \law(\bsY_s,\dots,\bsY_t)  \; , 
\end{align*}
where $\wto$ denotes weak convergence.  The sequence $\sequ{\bsY}$ is called the tail process of
$\sequ{\bsX}$. Furthermore, by \cite[Theorem~3.1]{basrak:segers:2009} the process $\sequ{\bsTheta}$
defined by $\bsTheta_j=\bsY_j/\norm{\bsY_0}$ is independent of $\norm{\bsY_0}$ and is called the
spectral tail process.

In the unpublished manuscript \cite{owada:samorodnitsky}, the tail measure of a regularly varying
time series was defined. It is the unique Borel measure $\nu$ on $\setE{d}$ with respect to the
product topology such that $\nu(\{\bszero\})=0$ and for all $s\leq t\in\Zset$
\begin{align*}
  \nu\circ \proj_{s,t}^{-1} = \nu_{s,t}
\end{align*}
on $(\Rset^d)^{[s,t]}\setminus\{\bszero\}$ where $\proj_{s,t}$ is the canonical projection of
$\setE{d}$ unto $(\Rset^d)^{[s,t]}$. It follows easily that the tail measure has the following
properties.
\begin{enumerate}[(i)]
\item \label{item:sigma-finite} $\nu$ is $\sigma$-finite;
\item \label{item:nu-shiftinvar} $\nu$ is shift invariant;
\item \label{item:nu-homogene} $\nu$ is homogeneous with index $-\alpha$, \ie\ $\nu(c \cdot) = c^{-\alpha} \nu$ for all $c>0$;
\item For every non negative measurable functional on $\setE{d}$,
  \begin{align}
    \esp[H(\bsY)] = \int_{\setE{d}} H(\bsy) \ind{\norm{\bsy_0}>1} \,\nu(\rmd \bsy) \; .  \label{eq:tail-tail}
  \end{align}
\end{enumerate}
The shift invariance of $\nu$ is a consequence of stationarity.  An alternate construction of the
tail measure (denoted $\mu^{\infty}$) was established in the more general framework of regular
variation on metric spaces in \cite{segers:zhao:meinguet:2017}.  Beyond its theoretical importance,
the tail measure is an extremely efficient tool to prove new results and give a much shorter
proof to known result.

The first such application will be in \Cref{sec:time-change} where we establish an alternative proof
of the time change formula (see (\ref{eq:time-shift-Theta})) which was first proved in
\cite{basrak:segers:2009} by using stationarity of the original time series and expressing the tail
process as a limit. Here, as in \cite{owada:samorodnitsky}, we will prove it using only the shift
invariance and homogeneity of the tail measure $\nu$.  Moreover, we will provide an equivalent
formulation of the formula (see \Cref{lem:time-shift-Y}) which turns out to be also very useful.

In \Cref{sec:tailprocesstozero}, we will restrict our attention to the case where the tail process
tends to zero at infinity. This property holds for most usual heavy tailed time series. It holds for
linear processes and for most Markov models of interest in time series such as GARCH-type processes
and solutions to stochastic recurrence equations.
Our first main result (\Cref{theo:tail-measure-infargmax}) will be that under this assumption, the
tail measure can be recovered from the spectral tail process conditioned to first achieve its
maximum at time zero. 

In \Cref{subsec:the-sequenceQ} we will introduce the sequence $\vbQ$ whose distribution is that of
the tail process standardized by its maximum, conditionally on the event that $|\bsY_0|$ is the
first exceedence of the tail process over 1. This sequence was introduced by
\cite{basrak:segers:2009} where it appears in the limiting theory for the point process of
exceedences and partial sums of the original time series and in the limits of the so-called cluster
functionals (which will be  introduced in \Cref{sec:emp-proc-clusters}). We will show that it can
also be used to recover the tail measure and is equivalent (in some sense to be made precise in
\Cref{prop:loi-Q}) to the spectral tail process conditioned to have its first maximum at time zero.

Cluster functionals have also been investigated in \cite{mikosch:wintenberger:2016} where, as a
consequence of using different techniques, expressions for their limits were obtained in terms of
the spectral tail process. These two different types of expressions for the limit of the same
quantities must therefore be equal but no direct proof of their equality had been given.  Moreover,
the sets of functionals for which limits have been obtained by one or the other method were not
equal. We will directly prove that these expressions are the same.
As a particular example, we will prove in \Cref{coro:equivalences} that
$\esp[(\sum_{j\in\Zset}|\vbQ_j|)^\alpha]<\infty$ if and only if
$\esp[(\sum_{j=0}^\infty|\vbTheta_j|)^{\alpha-1}]<\infty$.  This equivalence is of importance,
since those were the conditions under which limiting results were obtained in the literature, but it
was not known if these conditions were equivalent.

Previously, we will have analyzed in \Cref{subsec:extremal} the so-called candidate extremal index
$\vartheta = \pr(\sup_{j\geq1} \norm{\bsY_{j}}\leq1)$, introduced in \cite{basrak:segers:2009}, who
proved that it is positive under a condition on the original time series referred to as the
anticlustering condition (see~(\ref{eq:anticlustering})). Assuming only that
$\lim_{|j|\to\infty} |\bsY_j|=0$, we will prove that $\vartheta>0$. This is useful since the
ancticlustering condition, which is a standard assumption in the literature, is often much harder to
check than the convergence of the tail process to zero.

We will conclude \Cref{sec:tailprocesstozero} by extending and providing a very simple proof, based
on the tail measure,  of identities for quantities generalizing those introduced as cluster indices in
\cite{mikosch:wintenberger:2014}.

As already mentioned, the previous results are important in the context of limiting theory for heavy
tailed time series and are used to characterize the limits of cluster functionals. Such convergence
results were previously obtained by various methods and often by ad-hoc conditions for each
functional at hand. In \Cref{sec:emp-proc-clusters}, following \cite{basrak:planinic:soulier:2016},
we will consider clusters which are vectors of observations $(\vbX_1,\dots,\vbX_{r_n})$ of non
decreasing length $r_n$ as element of the space $\lo$ of shift equivalent sequences (see
\Cref{sec:emp-proc-clusters} for a precise definition).

In \cite{basrak:planinic:soulier:2016}, it is proved that the suitably normalized distribution of
the scaled clusters converge in the sense of $\mcm_0$ convergence of \cite{hult:lindskog:2006} under
the  anticlustering condition. We will prove in
\Cref{lem:anticlustering-useless} that if the tail process tends to zero at infinity, then the
cluster convergence mentioned above always holds for some sequence~$\{r_n\}$. This result also has
consequences for the convergence of the point process of clusters introduced in
\cite{basrak:planinic:soulier:2016} which generalizes the point process of exceedences and is a key
tool in the study of certain statistics and for establishing the (functional) convergence of the
partial sum process to a stable process when $\alpha\in(0,2)$.

We conclude the paper in \Cref{sec:maxstableprocesses} by recalling certain relations between
spectral tail processes and max-stable processes. In particular, when already given a non negative
process $\bsTheta$ satisfying the time change formula and $\lim_{|j|\to\infty} \Theta_j=0$, we
obtain an alternative construction, based on the tail meassure and \Cref{theo:tail-measure-infargmax},  
of a max-stable process whose spectral tail process is $\bsTheta$ to the one given recently in
\cite[Theorem~3.2]{janssen:2017}.

\paragraph{Notation} The following notation will be used throughout the paper. We use boldface letters for
vectors and sequences; for a sequence $\bsx=(\bsx_i)_{i\in\Zset} \in\setE{d}$, we write
$\bsx_{s,t}=(\bsx_s,\dots,\bsx_t)$, $\vbx^*_{s,t} = \max_{s\leq i \leq t} \norm{\bsx_i}$,
$\vbx^*= \max_{i\in\Zset} \norm{\bsx_i}$ and
$\norm[p]{\vbx} = (\sum_{j\in\Zset} \norm{\vbx_j}^p)^{1/p}$, $p>0$. 

Whenever convenient, we identity a vector $\vbx_{s,t}$ with $-\infty \leq s \leq t \leq +\infty$ to
an element of $(\Rset^d)^\Zset$ by completing it with zeros to the left if $s>-\infty$ or to the
right if $t<\infty$.

We consider the following
subspaces of $\setE{d}$: $\lzero=\{\vbx\in\setE{d}:\lim_{|j|\to\infty} \norm{\vbx_j}=0\}$ and for
$p>0$, $\lpspace{p}=\{\vbx\in\setE{d}:\norm[p]{\vbx}<\infty\}$.  

We denote by $\shift$ the backshift operator, \ie\ $(\shift\bsx)_j=\bsx_{j-1}$ and by $B^k$ its
$k$-th iterate for $k\in\Zset$.

A function $H:(\Rset^d)^\Zset\to \Rset$ is said to be homogeneous with degree $\alpha\in\Rset$ or
simply $\alpha$-homogeneous if $H(t\bsx) = t^\alpha H(\bsx)$ for all $\vbx\in (\Rset^d)^\Zset$ and
$t>0$, and it said to be shift invariant if $H(B\vbx)=H(\vbx)$ for all $\vbx\in (\Rset^d)^\Zset$. A
subset $A$ of $(\Rset^d)^\Zset$ is said to be homogeneous if $\bsx\in A$ implies $t\bsx\in A$ for
all $t>0$ and it is said to be shift invariant if $\bsx\in A$ if and only if $B\bsx\in A$.

\section{The time change formula}
\label{sec:time-change}
Since $\nu$ is homogeneous, it can be decomposed into ``radial and angular'' parts.  Define
$\setEstar{d} = \{\bsy\in\setE{d}:\norm{\bsy_0}>0\}$ and
$\setS{d} = \{\bsy\in\setE{d}:\norm{\bsy_0}=1\}$. Let $\psi$ be the map defined by
\begin{align*}
  \psi :(0,\infty) \times \setS{d} & \to \setEstar{d} \\
   (r,\bstheta) & \mapsto r\bstheta \; .
\end{align*}
Since $\nu(\{\bsy\in\setE{d}:\norm{\bsy_0}>1\}=1$, the measure
$\spectralmeasure$ defined by $\spectralmeasure=\nu(\{\bsy:|\bsy_0|>1,\; \bsy/|\bsy_0|\in \cdot\})$ is a probability
measure on $\setS{d}$ and by the homogeneity of $\nu$ it follows that (cf. \cite[Proposition 3.1,
Property (4)]{segers:zhao:meinguet:2017}
\[
\nu \circ \psi (dr,d\bstheta)=\alpha r^{-\alpha - 1} \rmd r \spectralmeasure (\rmd \bstheta) \; .
\]
Equivalently, if $H$ is a measurable $\nu$-integrable or nonnegative function on $\setEstar{d}$,
\begin{align*}
  \int_{\setEstar{d}} H(\bsy)\nu(\rmd\bsy) = \int_0^\infty \int_{\setS{d}} H(r\bstheta) 
  \spectralmeasure(\rmd \bstheta) \alpha r^{-\alpha-1} \rmd r \; .
\end{align*}
This formula can be extended to functions $H$ on $\setE{d}$ by adding the indicator
$\ind{\bsy_0\ne0}$, \ie
\begin{align}
  \int_{\setE{d}} H(\bsy) \ind{\bsy_0\ne0} \nu(\rmd\bsy) = \int_0^\infty \int_{\setS{d}} H(r\bstheta) 
  \spectralmeasure(\rmd \bstheta) \alpha r^{-\alpha-1} \rmd r \; .  \label{eq:polaire-tail-measure-0}
\end{align}
The indicator $\ind{\bsy_0\ne0}$ in the left hand side of~(\ref{eq:polaire-tail-measure-0}) cannot be
dispensed with since it is possible that $\nu(\{\bsy_0=0\})=\infty$.  If we denote by $\bsTheta$ a
random element on $\setE{d}$ with distribution $\spectralmeasure$, we obtain that $\vbTheta$ is the
spectral tail process of the time series $\{\vbX_j\}$, \ie\ 
\begin{align}
  \esp[H(\bsY)] & = \int_1^\infty \esp[ H(r\bsTheta)] \alpha r^{-\alpha-1} \rmd r \; . \label{eq:polar}
\end{align}

\begin{example}
  \label{xmpl:indicator-is-important}
  Let $\sequ{X}$ be a sequence of \iid\ nonnegative regularly varying random variables with tail
  index $\alpha>0$. Then the tail process and the spectral tail process are trivial:
  $Y_j=\Theta_j=0$ for all $j\ne0$. Consider the function $H(\bsy) = \ind{y_1>1}$. Then
  \begin{align*}
    \int_0^\infty \esp[ H(r\bsTheta)] \alpha r^{-\alpha-1} \rmd r = \int_0^\infty \esp[\ind{r\Theta_1>1}] \alpha r^{-\alpha-1} \rmd r = 0 \; .
  \end{align*}
  However, because of shift invariance,
  \begin{align*}
    \int_{\setE{d}} H(\bsy)  \nu(\rmd\bsy)  = \int_{\setE{d}} \ind{y_0>1} \nu(\rmd\bsy) = 1 \; .
  \end{align*}
  This illustrates the necessity of the indicator in the left hand side of
  (\ref{eq:polaire-tail-measure-0}).
\end{example}

We now obtain and prove a new version of the time change formula of \cite{basrak:segers:2009}.
\begin{lemma}
  \label{lem:time-shift-Y}
  Let $H$ be a non negative measurable functional on $\setE{d}$. Then, for all $k\in\Zset$ and $t>0$,
  \begin{align}
    \esp[H(\shift^k\bsY)\ind{\norm{\bsY_{-k}}>t}] = t^{-\alpha} \esp[H(t\bsY)\ind{\norm{\bsY_k}>1/t}] \;
    . \label{eq:time-shift-Y}
  \end{align}
\end{lemma}

\begin{proof}
  Applying (\ref{eq:tail-tail}), the homogeneity and  shift invariance of $\nu$ yields
  \begin{align*}
    \esp[H(\shift^k\bsY)\ind{\norm{\bsY_{-k}}>t}] 
    & = \int_{\setE{d}} H(\shift^k\bsy) \ind{\norm{\bsy_{0}}>1} \ind{\norm{\bsy_{-k}}>t} \,\nu(\rmd\bsy)  \\
    & = \int_{\setE{d}} H(\bsy) \ind{\norm{\bsy_{k}}>1} \ind{\norm{\bsy_{0}}>t} \,\nu(\rmd\bsy)  \\
    & = t^{-\alpha} \int_{\setE{d}} H(t\bsx) \ind{\norm{\bsx_{k}}>1/t} \ind{\norm{\bsx_{0}}>1} \,\nu(\rmd\bsx)  \\
    & = t^{-\alpha} \esp[H(t\bsY)\ind{\norm{\bsY_k}>1/t}] \; .
  \end{align*}
\end{proof}

By an application of the polar decomposition (\ref{eq:polaire-tail-measure-0}), it is easily seen
that (\ref{eq:time-shift-Y}) is equivalent to the time change formula of \cite{basrak:segers:2009}:
\begin{align}
  \esp[H(\shift^k\bsTheta) \ind{\norm{\bsTheta_{-k}}\ne0}] = \esp[H(\norm{\bsTheta_k}^{-1}\bsTheta) \norm{\bsTheta_k}^\alpha] \;
  . \label{eq:time-shift-Theta}
\end{align}
where the quantity inside the expectation on the right hand side is understood to be $0$ when
$|\vbTheta_k|=0$.  A proof of this equivalence is in the appendix.

\begin{remark}
  Note that (\ref{eq:time-shift-Theta}) was proved in \cite[Theorem~3.1]{basrak:segers:2009} by
  using the definition of the tail process as a limit and therefore restricting it to continuous
  functions. The present proof is without such restriction and arguably more straightforward. 
\end{remark}

\begin{remark}
  If $H$ is homogeneous with degree $0$, then~(\ref{eq:time-shift-Theta}) yields for all
  $k\in\Zset$,
  \begin{align}
    \esp[H(\shift^k\bsTheta)\ind{\norm{\bsTheta_{-k}}\ne0}] = \esp[H(\bsTheta)\norm{\vbTheta_k}^\alpha] \; . \label{eq:time-shift-homo-0}
  \end{align}
  Conversely, by considering the function $\vbx\mapsto H(\norm{\vbx_k}^{-1}\vbx)$ it is easily
  seen that (\ref{eq:time-shift-homo-0}) is actually equivalent to (\ref{eq:time-shift-Theta}). If
  $H$ is homogeneous with degree $\alpha$, then~(\ref{eq:time-shift-Theta}) yields for all
  $k\in\Zset$,
  \begin{align}
    \esp[H(\shift^k\bsTheta)\ind{\norm{\bsTheta_{-k}}\ne0}] = \esp[H(\bsTheta)\ind{\norm{\bsTheta_k}\ne0}] \; . \label{eq:time-shift-homo-alpha}
  \end{align}
  If moreover $\sum_{i\in\Zset}\pr(\norm{\vbTheta_i}=0)=0$, then we obtain, for all $k\in\Zset$,
  \begin{align}
    \esp[H(\shift^k\bsTheta)] = \esp[H(\bsTheta)] \; . \label{eq:time-shift-Theta-positive-homo-alpha}
  \end{align}
  This property deceptively looks like stationarity, but it is only valid for functionals $H$ which
  are homogeneous with degree $\alpha$ and if $\pr(\norm{\vbTheta_k}=0)=0$ for all $k\in\Zset$.
\end{remark}

The shift invariance and homogeneity of $\nu$ allow to relate the null shift-invariant homogeneous
sets for $\nu$ and for the distribution of $\vbY$.

\begin{lemma}
  \label{lem:0-infty}
  Let $A$ be a shift invariant, homogeneous  measurable set in $\setE{d}$. Then
    $\nu(A)\in\{0,\infty\}$ and the following statements are equivalent: (i) $\nu(A)=0$; (ii)
  $\pr(\vbY\in A)=0$; (iii) $\pr(\vbTheta\in A)=0$.  
\end{lemma}

\begin{proof}
  Since $\nu(\{\bszero\})=0$, we have, by the shift invariance of $\nu$ and $A$,
  \begin{align}
    \nu(A \cap\{\norm{\vby_0}>0\}) \leq \nu(A) \leq \sum_{j\in\Zset} \nu(A \cap\{\norm{\vby_j}>0\}) 
    = \sum_{j\in\Zset} \nu(A \cap\{\norm{\vby_0}>0\}) \; . \label{eq:nu-decompose}
  \end{align}
  Applying the homogeneity and shift invariance of $\nu$ and $A$, the monotone convergence theorem
  and the definition of~$\vbY$, we obtain
  \begin{align*}
    \nu(A \cap\{\norm{\vby_0}>0\}) 
    & = \lim_{\epsilon\to 0} \nu(A \cap\{\norm{\vby_0}>\epsilon\})  \\
    & = \lim_{\epsilon\to 0} \epsilon^{-\alpha} \nu(A \cap\{\norm{\vby_0}>1\})  = \lim_{\epsilon\to 0} \epsilon^{-\alpha} \pr(\vbY\in A) \; . 
  \end{align*}
  This proves that $\nu(A \cap\{\norm{\vby_0}>0\}) =0$ if and only if $\pr(\vbY\in A) =0$ and that
  $\nu(A \cap\{\norm{\vby_0}>0\}) =\infty$ if $\nu(A \cap\{\norm{\vby_0}>0\}) >0$. It now follows
  from (\ref{eq:nu-decompose}) that $\nu(A)\in\{0,\infty\}$ and that the statements (i) and (ii) are
  equivalent. To finish the proof it just remains to notice that, since $A$ is homogeneous,
  $\pr(\vbTheta\in A)=\pr(\vbY\in A)$.
\end{proof}

\section{Properties of the tail process when $\lim_{j\to\infty} \norm{\vbY_j}=0$} 
\label{sec:tailprocesstozero}

In this section, we restrict our attention to tail processes which satisfy the following
condition.
\begin{align}
  \pr \left(  \lim_{|k|\to\infty} \norm{\bsY_k} = 0 \right) = 1 \; .
  \label{eq:Y_to_zero}
\end{align}
This condition is satisfied by most time series models of interest. It will be further discussed in
\Cref{sec:emp-proc-clusters}; here we simply admit it as our working assumption.  By
\Cref{lem:0-infty}, the property (\ref{eq:Y_to_zero}) means that the tail measure $\nu$ is supported
on the shift invariant and homogeneous set $\lzero=\{\lim_{|j|\to\infty} \norm{\vby_j}=0\}$.

\subsection{Recovering the tail measure}
An important consequence of~(\ref{eq:Y_to_zero}) is that \as\ $\vbY^*<\infty$ and there is a first
time index at which the maximum for the sequence $\vbY$ is achieved. To formalize this remark, we
introduce the infargmax functional $I$, defined on $(\Rset^d)^\Zset$ by
\begin{align*}
  I(\vby) = 
  \begin{cases}
    & j \in\Zset \mbox{ if } \vby_{-\infty,j-1}^* < \norm{\vby_j} \mbox{ and } \vby_{j+1,\infty}^*
    \leq \norm{\vby_j} \; , \\
    & -\infty \mbox{ if } \vby^* = \vby_{-\infty,j}^* \mbox{ for all } j \in \Zset \\ 
    & +\infty \mbox{ if }     \vby^* > \vby_{-\infty,j}^* \mbox{ for all } j \in \Zset \; .
  \end{cases}
\end{align*}
For instance, the infargmax of a constant sequence is $-\infty$. The infargmax is achieved in
$\Zset$ if there exists a first time when the maximum is achieved.  The event $I(\vby)\in\Zset$ can
be expressed as
\begin{align*}
  \sum_{j\in\Zset} \ind{I(\vby)=j}=1 \; .   
\end{align*}
By \Cref{lem:0-infty}, we  have
\begin{align*}
  \nu(\{I(\vby)\notin\Zset\})=0 \Longleftrightarrow \pr(I(\vbTheta)\in\Zset)=1 \; .
\end{align*}

\begin{theorem}
  \label{theo:tail-measure-infargmax}
  Assume that $\pr(I(\vbTheta)\in\Zset)=1$.  Then $\pr(I(\vbTheta)=0)>0$ and for all non negative
  measurable functions~$H$,
  \begin{align}
    \label{eq:tailmeasure-infargmax}
    \nu(H) =   \sum_{j\in\Zset} \int_0^\infty \esp[ H(r\shift^j\vbTheta) \ind{I(\vbTheta)=0}] \alpha r^{-\alpha-1} \rmd r \; .
  \end{align}
  If moreover $\sum_{j\in\Zset} \nu(\{\norm{\vby_j}=0\})=0$, or equivalently
  $\sum_{j\in\Zset}\pr(\norm{\vbTheta_j}=0)=0$, then
  \begin{align}
    \nu(H) =  \int_0^\infty \esp[H(r\vbTheta)] \alpha r^{-\alpha-1} \rmd r \; . \label{eq:tailmeasure-Theta-nonnul}
  \end{align}
\end{theorem}

\begin{proof}
  Let $H$ be a non negative measurable function on $\setE{d}$. Since $\nu(\{I(\vby)\notin\Zset\})=0$
  by assumption, the shift invariance of $\nu$ yields
  \begin{align*}
    \nu(H) = \sum_{j\in\Zset} \int_{\setE{d}} H(\vby) \ind{I(\vby)=j} \nu(\rmd \vby)
    = \sum_{j\in\Zset} \int_{\setE{d}} H(\shift^j\vby)  \ind{I(\vby)=0} \nu(\rmd \vby) \; .
  \end{align*}
  Since $I(\vby)=0$ implies that $\norm{\vby_0}>0$, applying the polar
  decomposition~(\ref{eq:polaire-tail-measure-0}) to the function
  $\vby\mapsto\sum_{j\in\Zset} H(\shift^j\vby) \ind{I(\vby)=0} $
  yields~(\ref{eq:tailmeasure-infargmax}).

  In the case $\pr(\norm{\vbTheta_i}=0)=0$ for all $i\in\Zset$, we can apply the time change
  formula~(\ref{eq:time-shift-Theta-positive-homo-alpha}) to the function
  $L(\vby) = \int_0^\infty H(r\vby) \alpha r^{-\alpha-1} \rmd r$ which is
  homogeneous with degree $\alpha$ and we obtain
  \begin{align*}
    \nu(H) & = \sum_{j\in\Zset} \esp[L(\shift^j\vbTheta) \ind{I(\vbTheta)=0}] 
             = \sum_{j\in\Zset} \esp[L(\vbTheta) \ind{I(\shift^{-j}\vbTheta)=0}] \\ 
           & = \sum_{j\in\Zset} \esp[L(\vbTheta) \ind{I(\vbTheta)=j}]  = \esp[L(\vbTheta)] \; .
  \end{align*}
  This proves~(\ref{eq:tailmeasure-Theta-nonnul}).  Taking $H(\vby)=\ind{\norm{\vby_0}>1}$,
  (\ref{eq:tailmeasure-infargmax}) yields
  \begin{align}
    1 & = \nu(\{\norm{\vby_0}>1\}) \nonumber
        = \sum_{j\in\Zset} \int_0^\infty \pr( r\norm{\vbTheta_j}>1,I(\vbTheta)=0)  \alpha r^{-\alpha-1} \rmd r \\
      & = \sum_{j\in\Zset} \esp[ \norm{\vbTheta_j}^\alpha \ind{I(\vbTheta)=0}] \; . \label{eq:summability-infargmax}
  \end{align}
  This proves that $\pr(I(\vbTheta)=0)>0$. 
\end{proof}
Using the representation of the tail measure, we can further refine \Cref{lem:0-infty}.  

\begin{corollary}
  \label{coro:nul-infargmax}
  Assume that $\pr(I(\vbTheta)\in\Zset)=1$. If $A$ is shift invariant and homogeneous, then the
  following statements are equivalent
  \begin{enumerate}[(i)]
  \item \label{item:nuA0} $\nu(A)=0$;
  \item \label{item:thetaA0}  $\pr(\vbTheta \in A)=0$;
  \item \label{item:thetaA0-conditional}  $\pr(\vbTheta \in A \mid I(\vbTheta)=0)=0$.
  \end{enumerate}
\end{corollary}
\begin{proof}
  The equivalence between~\ref{item:nuA0} and~\ref{item:thetaA0} was already stated in
  \Cref{lem:0-infty}. The equivalence between~\ref{item:nuA0} and~\ref{item:thetaA0-conditional} follows from~(\ref{eq:tailmeasure-infargmax}).  
\end{proof}
The equivalence \ref{item:thetaA0} and~\ref{item:thetaA0-conditional} is useful in practice since as
we will see right below it may be  easier to prove that an event has a probability zero
conditionally on the first maximum being achieved at time zero than unconditionally.

The following result has been proved in the related context of max-stable processes by
\cite{dombry:kabluchko:2016} and recently by \cite{janssen:2017}. We provide an alternate
straightforward proof.  See also.  \cite{janssen:2017}.
\begin{corollary}
  \label{lem:equivalence-limit-summability}
  The following statements are equivalent.
  \begin{enumerate}[(i)]
  \item \label{item:infargmax} $\pr(I(\vbTheta)\in\Zset)=1$;
  \item \label{item:limit} $\pr(\lim_{|j|\to\infty} \norm{\vbTheta_j}=0)=1$;
  \item \label{item:sum}  $\pr \left(\sum_{j\in\Zset} \norm{\vbTheta_j}^\alpha < \infty \right) = 1$.
  \end{enumerate}
\end{corollary}

\begin{proof}
  The implications \ref{item:sum}$\implies$\ref{item:limit} and
  \ref{item:limit}$\implies$\ref{item:infargmax} are obvious.  We only need to prove the
  implication \ref{item:infargmax}$\implies$\ref{item:sum}.  By \Cref{theo:tail-measure-infargmax},
  \ref{item:infargmax} implies the identity~(\ref{eq:summability-infargmax}) which implies
  that
  \begin{align*}
    \pr \left(\sum_{j\in\Zset} \norm{\vbTheta_j}^\alpha < \infty \mid I(\vbTheta)=0\right) = 1 \; .
  \end{align*}
  By \Cref{coro:nul-infargmax}, this proves~\ref{item:sum}.
\end{proof}

\Cref{lem:equivalence-limit-summability} yields the following property which has been used in the
literature, but to the best of our knowledge, never proved.  If
$\lim_{|k|\to\infty} \norm{\vbY_k}=0$ and $\alpha\leq1$, then
\begin{align}
  \pr\left( \sum_{j\in\Zset} |\bsTheta_j|<\infty \right) = 1 \; .  \label{eq:summability-theta}
\end{align}

\subsection{The candidate extremal index}
\label{subsec:extremal}
Following \cite{basrak:segers:2009}, we define 
\begin{align}
  \vartheta = \pr\left(\sup_{j\geq1} \norm{\bsY_{j}}\leq1\right) \; . \label{eq:def-candidate-extremal}
\end{align}
In terms of the tail measure, we have
\begin{align}
  \vartheta = \nu(\{\bsy_{1,\infty}^*\leq1,\norm{\bsy_0}>1\}) \; . \label{eq:def-candidate-tailmeasure}
\end{align}
The candidate extremal index turns out to be the true extremal index of many time series models. The
relation between the candidate and true extremal index will be further developed in
\Cref{sec:emp-proc-clusters,sec:maxstableprocesses}.  Decomposing the event
$\{\sup_{j\geq 1}\norm{\bsY_{j}}>1\}$ according to the first time the tail process is greater than
1 and applying the time change formula (\ref{eq:time-shift-Y}) (with $t=1$), we obtain
\begin{align*}
  \pr \left(\sup_{j\geq 1}\norm{\bsY_{j}}>1\right)
  & = \sum_{k \geq 1} \pr\left(\max_{1\leq j\leq k-1}\norm{\bsY_{j}}\leq 1,\norm{\bsY_{k}}>1\right) \\
  & = \sum_{k \geq 1} \pr\left(\max_{-k+1\leq j\leq -1}\norm{\bsY_{j}}\leq 1,\norm{\bsY_{-k}}>1\right)
    = \pr\left(\sup_{j\leq -1}\norm{\bsY_{j}}>1\right) \; .
\end{align*}
Thus the candidate extremal index can be expressed as the probability that the first exceedence over
1 happens at time 0:  
\begin{align}
  \vartheta = \pr\left(\sup_{j\leq -1}\norm{\bsY_{j}}\leq1\right) \; .   \label{eq:altcandidate-BS09} 
\end{align}
This identity was obtained \cite{guivarch:lepage:2016} for solutions to stochastic recurrence
equations. For general time series, \cite{basrak:segers:2009} proved that
\eqref{eq:altcandidate-BS09} holds under the so-called anticlustering condition (see (\ref{eq:anticlustering}) below) by
using the original time series. Furthermore, it is also proved in the same manner in
\cite{basrak:segers:2009} that the anticlustering condition
implies that $\vartheta>0$. Since the candidate extremal index is defined in terms of the tail
process only, it is natural to give a proof using only the tail process under the assumption that
the tail process tends to zero.
\begin{lemma}
  \label{lem:candidat>0} 
  If $ \pr(\lim_{|k|\to\infty} \norm{\bsY_k} = 0)=1$, then $\vartheta>0$. 
\end{lemma}
\begin{proof}
  Since $\pr(\norm{\vbY_0}>1)=1$ and by assumption there is always a last time when the tail process
  is bigger than $1$, applying the time change formula~(\ref{eq:time-shift-Y}), we have 
  \begin{align*}
    1 & = \pr(\vbY_{0,\infty}^*>1)= \sum_{j\geq 0} \pr(\vbY_{j+1,\infty}^*\leq1,\norm{\vbY_j}>1) \\
      & = \sum_{j\geq 0} \pr(\vbY_{1,\infty}^*\leq1,\norm{\vbY_{-j}}>1) 
        \leq \sum_{j\geq0} \pr(\vbY_{1,\infty}^*\leq1)  = \infty \times \vartheta  \; .
  \end{align*}
  This implies that $\vartheta>0$.
\end{proof}

\subsection{The sequence $\vbQ$}
\label{subsec:the-sequenceQ}

From now on, we assume that $\pr(\lim_{|j|\to\infty} |\vbY_j|=0)=1$, which ensures that
$\vartheta>0$. 
\begin{definition}[\cite{basrak:segers:2009}]
  \label{def:def-sequenceQ}
  The sequence $\vbQ=\{\vbQ_j,j\in\Zset\}$ is a random sequence whose distribution is that of
  $(\bsY^*)^{-1}\bsY$ (or $(\vbTheta^*)^{-1}\vbTheta$) conditionally on \mbox{$\vbY_{-\infty,-1}^*\leq1$}.
\end{definition}
The sequence $\vbQ$ appears in limits of so-called cluster functionals.  This will be further
developed in \Cref{sec:emp-proc-clusters}. Here we study it formally.  We first show that this
sequence is closely related to the sequence $\vbTheta$ conditioned to have its first maximum at $0$.

\begin{proposition}
  \label{prop:loi-Q}
  Assume that $\pr(\lim_{|k|\to\infty} \norm{\vbY_k}=0)=1$.
  Let $H$ be a shift invariant, non negative measurable function on $\setE{d}$. Then
  \begin{align}
    \vartheta \esp[H(\vbQ)] = \esp[H(\vbTheta)\ind{I(\vbTheta)=0}] \; .  \label{eq:identity-Q-infargmax}
  \end{align}
\end{proposition}

\begin{proof}
  Let $K$ be a non negative measurable shift invariant functional.  Applying the time change
  formula (\ref{eq:time-shift-Y}) yields
  \begin{align}
    \esp[K(\vbY)  \ind{\vbY_{-\infty,-1}^*\leq1}] 
    & = \sum_{j\in\Zset} \esp[K(\vbY) \ind{\vbY_{-\infty,-1}^*\leq1}\ind{I(\vbY)=j}] \nonumber  \\
    & = \sum_{j\in\Zset} \esp[K(\vbY) \ind{\vbY_{-\infty,-1}^*\leq1}\ind{I(\vbY)=j}\ind{\norm{\vbY_j}>1}] \nonumber  \\
    & = \sum_{j\in\Zset} \esp[K(\vbY) \ind{\vbY_{-\infty,-j-1}^*\leq1}\ind{I(\shift^j\vbY)=j}\ind{\norm{\vbY_{-j}}>1}] \nonumber  \\
    & = \esp[K(\vbY) \ind{I(\vbY)=0} \sum_{j\in\Zset} \ind{\vbY_{-\infty,-j-1}^*\leq1} \ind{\norm{\vbY_{-j}}>1}] \nonumber \\
    & = \esp[K(\vbY) \ind{I(\vbY)=0}] \; .  \label{eq:identity-Y/Y*}
  \end{align}
  Applying this identity to the function $K(\vby) = H(\vby/\vby^*)$ 
  yields~(\ref{eq:identity-Q-infargmax}).
\end{proof}

Let $H$ be a shift invariant non negative measurable functional on $\setE{d}$. Applying the
identity~(\ref{eq:identity-Y/Y*}) we obtain for $t\geq1$, 
\begin{align}
  \esp[H(\vbY/\vbY^*)\ind{\vbY^*>t}\mid \vbY_{-\infty,-1}^*\leq1] 
  & = \vartheta^{-1}   \esp[H(\vbTheta) \ind{\norm{\vbY_0}>t}  \ind{I(\vbTheta)=0} ] \nonumber \\
  & = \vartheta^{-1}   t^{-\alpha} \esp[H(\vbTheta)   \ind{I(\vbTheta)=0} ] \nonumber \\
  & = t^{-\alpha}      \esp[H(\vbQ)] \; . \label{eq:identity-Y-Ystar}
\end{align} 
The previous result implies that conditionally on $\vbY_{-\infty,-1}^*\leq1$, $\vbY^*$ has the same
Pareto distribution as $\norm{\vbY_0}$ (unconditionally), but it does not imply that $\vbY^*$ and
$(\vbY^*)^{-1}\vbY$ are independent, since (\ref{eq:identity-Y-Ystar}) holds only for shift
invariant functionals.  However, the latter statement is true if one considers $(\vbY^*)^{-1}\vbY$
as a random element in the space $\lo$ of shift-equivalent sequences. This will be further developed
in \Cref{sec:emp-proc-clusters}.  These results were originally proved in \cite{basrak:tafro:2015}
by different means under the anticlustering condition (\ref{eq:anticlustering}).  The present proof,
which assumes only $\pr(\lim_{|j|\to\infty} |\vbY_j|=0)=1$, is  simpler and moreover shows that
these results are direct consequences of the fact that $\norm{\vbY_0}$ is Pareto distributed and
independent of the spectral tail process.

Anoter important consequence of \Cref{prop:loi-Q} is that the tail measure can also be recovered
from the sequence $\vbQ$. For a non negative measurable $H$, by \Cref{theo:tail-measure-infargmax}
we have
\begin{align}
  \nu(H) = \vartheta \sum_{j\in\Zset} \int_0^\infty \esp[ H(r\shift^j\vbQ)] \alpha r^{-\alpha-1} \rmd r \; .    \label{eq:tailmeasure-Q}
\end{align}

Applying (\ref{eq:tailmeasure-Q}) to the function $H:\vbx\mapsto \ind{|\vbx_0|>1}$ yields
\begin{align}
  1 = \nu(H) =   \vartheta \sum_{j\in\Zset} \esp[\norm{\vbQ_j}^\alpha]  \; .  \label{eq:sumQtheta-egal1}
\end{align}
The inequality $ \vartheta \sum_{j\in\Zset} \esp[\norm{\vbQ_j}^\alpha] \leq1$ was obtained in
\cite[Theorem~2.6]{davis:hsing:1995} by an application of Fatou's lemma. It was also stated there
that equality holds under an additional uniform integrability assumption. We have thus proved that
(\ref{eq:sumQtheta-egal1}) holds without any additional assumption and this seems to be new.

We now prove more identities between quantities expressed in terms of the spectral tail process or
in terms of the sequence $\vbQ$. The equality of some of these quantities was already indirectly
known, since they appeared as limits of the same quantities, but obtained through different
methods. For some of them, the case $\alpha=1$ had not yet been treated.  We first prove an identity
which will be the main path from $\vbQ$ to $\vbTheta$.

\begin{lemma}
  \label{lem:newidentity}
  Assume that $\pr(\lim_{|j|\to\infty}\norm{\vbY_j}=0)=1$.  Let $H$ be a non negative, shift invariant, 
  $\alpha$-homogeneous measurable function on~$\setE{d}$. Then,
  \begin{align}
    \vartheta \esp[H(\vbQ)] = \esp[H(\vbTheta)\ind{I(\vbTheta)=0}] 
    = \esp \left[ \frac{H(\vbTheta)}{\norm[\alpha]{\vbTheta}^\alpha}\right] \; . \label{eq:newidentity}
  \end{align}

\end{lemma}

\begin{proof}
  The first equality in~(\ref{eq:newidentity}) is a repeat of (\ref{eq:identity-Q-infargmax}).  The
  function $\vby \mapsto \norm[\alpha]{\bsy}^{-\alpha} H(\bsy)$ is shift invariant and homogeneous
  with degree 0, thus applying the time change formula~(\ref{eq:time-shift-homo-0}), we obtain
  \begin{align*}
    \esp[H(\vbTheta)\ind{I(\vbTheta)=0}] 
    & = \sum_{j\in\Zset} \esp \left[ |\vbTheta_{j}|^\alpha  \frac{H(\vbTheta)}{\norm[\alpha]{\vbTheta}^\alpha} \ind{I(\vbTheta)=0} \right]  \\
    & = \sum_{j\in\Zset} \esp \left[ \frac{H(\vbTheta)}{\norm[\alpha]{\vbTheta}^\alpha} \ind{I(B^j\vbTheta)=0} \right]  \\
    & =  \esp \left[ \frac{H(\vbTheta)}{\norm[\alpha]{\vbTheta}^\alpha} \sum_{j\in\Zset}\ind{I(\vbTheta)=-j} \right] 
      =  \esp \left[ \frac{H(\vbTheta)}{\norm[\alpha]{\vbTheta}^\alpha}  \right] \; .
  \end{align*}
  We have used in the middle lines that $I(\shift^j\vbTheta)=0$ implies $|\bsTheta_{-j}|\ne0$ and is
  equivalent to $I(\vbTheta)=-j$ and to conclude we used that $\pr(I(\vbTheta)\in\ZZ)=1$ as a
  consequence of the assumption $\pr(\lim_{|j|\to\infty}\norm{\vbY_j}=0)=1$.
\end{proof}
As a consequence of \Cref{lem:newidentity}, for every measurable shift invariant
$\alpha$-homogeneous function $H$ on $(\Rset^d)^\Zset$, the quantities
\begin{align*}
  \esp[|H(\vbQ)|]  \; ,   \  \     \esp[|H(\vbTheta)|\ind{I(\vbTheta)=0}]  \; ,   \  \ 
  \esp\left[\frac{|H(\vbTheta)|}{\norm[\alpha]{\vbTheta}^\alpha}\right] \; , 
\end{align*}
are simultaneously finite or infinite and in the former case, the identity (\ref{eq:newidentity})
holds.

The previous identities involve the sequence $\vbQ$ and the sequence $\vbTheta$. In practice, the
quantities in terms of the spectral tail process are usually easier to compute explicitly or for use
in simulation since they do not involve a conditioning contrary to those with the sequence~$\vbQ$.
Moreover, it is often relatively easy to compute the forward tail process $\{\vbTheta_j,j\geq0\}$ but more
difficult to compute the backward tail process $\{\vbTheta_j,j<0\}$. Therefore, it is useful to find
an expression of the previous quantities in terms of the forward tail process alone if possible. The
following result gives sufficient conditions for such an identity to hold. The conditions are
unprimitive but easy to check in examples.

\begin{lemma}
  \label{lem:general-identity-forward}
  Assume that $\pr(\lim_{|j|\to\infty} |\vbY_j|=0)=1$. Let $H$ be a shift invariant,
  $\alpha$-homogenous measurable function defined on a subset $\mco$ of $\ell_\alpha$ such that
  $\pr(\vbTheta_{n,\infty}\in\mco)=1$ for all $n\in \Zset$.  Assume moreover
  \begin{enumerate}[(i)]
  \item \label{item:integrabilite-diff} $\EE[|H(\vbTheta_{0,\infty}) - H(\vbTheta_{1,\infty})|] < \infty$;
  \item \label{item:to0} $\pr(\lim_{n\to\infty} H(\vbTheta_{n,\infty}) = 0)=1$;
  \item \label{item:toH}  $\pr\left(\lim_{n\to\infty} H(\vbTheta_{-n,\infty}) = H(\vbTheta)\right)=1$. 
  \end{enumerate}
  Then $\esp[|H(\vbQ)|]<\infty$ and 
  \begin{align}
    \label{eq:double-identity-general}
    \vartheta \esp[H(\vbQ)]     = \EE[H(\vbTheta_{0,\infty}) - H(\vbTheta_{1,\infty})] \; .
  \end{align}

\end{lemma}

\begin{proof}
  Since the function
  $\vbx\mapsto \norm[\alpha]{\vbx}^{-\alpha} \{H(\vbx_{0,\infty}) - H(\vbx_{1,\infty})\}$ is
  $\alpha$-homogeneous and equal to $0$ whenever $|\vbx_0|=0$, by the time change formula
  (\ref{eq:time-shift-homo-0}), we have
  \begin{align*}
    \sum_{j\in\Zset} \esp\left[ \frac{|H(\vbTheta_{-j,\infty}) - H(\vbTheta_{-j+1,\infty})|}{\norm[\alpha]{\vbTheta}^\alpha} \right]  
    & = \sum_{j\in\Zset} \esp\left[ |\vbTheta_j|^\alpha \frac{|H(\vbTheta_{0,\infty}) - H(\vbTheta_{1,\infty})|}{\norm[\alpha]{\vbTheta}^\alpha} \right]  \\
    & = \EE[|H(\vbTheta_{0,\infty}) - H(\vbTheta_{1,\infty})|] < \infty \; .
  \end{align*}
  Consequently, $\pr(\sum_{j\in\Zset} |H(\vbTheta_{-j,\infty}) - H(\vbTheta_{-j+1,\infty})|<\infty)=1$ and 
  \begin{align*}
    \sum_{j\in\Zset} \esp\left[ \frac{H(\vbTheta_{-j,\infty}) - H(\vbTheta_{-j+1,\infty})}{\norm[\alpha]{\vbTheta}^\alpha} \right]  
    & = \esp\left[\sum_{j\in\Zset} \frac{H(\vbTheta_{-j,\infty}) - H(\vbTheta_{-j+1,\infty})}{\norm[\alpha]{\vbTheta}^\alpha} \right]   \\
    & = \EE[H(\vbTheta_{0,\infty}) - H(\vbTheta_{1,\infty})]  \; .
  \end{align*}
  On the other hand, assumptions~\ref{item:to0} and~\ref{item:toH} and the dominated convergence
  theorem ensure that
  \begin{align*}
    \sum_{j\in\Zset} \frac{H(\vbTheta_{-j,\infty}) - H(\vbTheta_{-j+1,\infty})}{\norm[\alpha]{\vbTheta}^\alpha}   
    & =  \lim_{n\to\infty} \sum_{-n<j\leq n} \frac{H(\vbTheta_{-j,\infty}) - H(\vbTheta_{-j+1,\infty})}{\norm[\alpha]{\vbTheta}^\alpha}\\
    & =  \lim_{n\to\infty} \frac{H(\vbTheta_{-n,\infty})-H(\vbTheta_{n,\infty})}{\norm[\alpha]{\vbTheta}^\alpha} 
      =  \frac{H(\vbTheta)}{\norm[\alpha]{\vbTheta}^\alpha}\; .
  \end{align*}  
  Hence, $\esp[ \norm[\alpha]{\vbTheta}^{-\alpha} |H(\vbTheta)|] < \infty$ and
  $\esp[\norm[\alpha]{\vbTheta}^{-\alpha}H(\vbTheta) ] = \EE[H(\vbTheta_{0,\infty}) -
  H(\vbTheta_{1,\infty})]$.
  \Cref{lem:newidentity} finally yields that $\esp[|H(\vbQ)|]<\infty$ and
  that~(\ref{eq:double-identity-general}) holds.
\end{proof}

\begin{example}
  As a first illustration of the previous results, we provide other expressions for the candidate
  extremal index $\vartheta$. These expressions might be used for statistical inference on the
  extremal index when it is known to be equal to the candidate  extremal index.  If
  $ \pr(\lim_{|k|\to\infty} \norm{\bsY_k} = 0)=1$, then
  \begin{align}
    \vartheta 
    & = \pr(I(\vbY)=0) = \pr(I(\vbTheta)=0)  \label{eq:candidate-infargmax} \\
    & = \esp \left[ \frac{(\vbY^*)^\alpha}{\sum_{j\in\Zset} \norm{\vbY_j}^\alpha} \right] 
      = \esp \left[ \frac{(\vbTheta^*)^\alpha}{\sum_{j\in\Zset} \norm{\vbTheta_j}^\alpha} \right] \label{eq:EI-debicky-hashorva} \\
    & = \esp \left[ (\vbTheta_{0,\infty}^*)^\alpha - (\vbTheta^*_{1,\infty})^\alpha \right] \label{eq:vartheta-diff} \\
    & = \esp \left[ \frac{1}{\sum_{j\in\Zset} \ind{|\vbY_j|>1}} \right] \label{eq:candidate-hashorva}
  \end{align}
  The identity (\ref{eq:candidate-infargmax}) is obtained by applying \ref{prop:loi-Q} to the
  function $H\equiv1$; (\ref{eq:EI-debicky-hashorva}) is obtained by applying \Cref{lem:newidentity}
  to the function $H$ defined by $H(\vbx) = \vbx^*$ and~(\ref{eq:vartheta-diff}) is obtained by
  applying \Cref{lem:general-identity-forward} to the same function.  We only need to
  prove~(\ref{eq:candidate-hashorva}). The assumption implies that
  $\pr(\sum_{i\in\Zset} \ind{|\vbY_i|>1}<\infty)=1$.  Applying the time change
  formula~(\ref{eq:time-shift-Y}), we obtain
  \begin{align*}
    \vartheta 
    & = \pr(\vbY_{-\infty,-1}^* \leq 1) = \esp \left[ \frac{\sum_{j\in\Zset} \ind{|\vbY_j|>1}}{\sum_{i\in\Zset} 
      \ind{|\vbY_i|>1}} \ind{\vbY_{-\infty,-1}^* \leq 1} \right] \\
    & = \sum_{j\in\Zset} \esp \left[ \frac{\ind{\vbY_{-\infty,-1}^* \leq 1} \ind{|\vbY_j|>1}}{\sum_{i\in\Zset} 
      \ind{|\vbY_i|>1}}  \right] \\
    & = \sum_{j\in\Zset} \esp \left[ \frac{ \ind{\vbY_{-\infty,-j-1}^* \leq 1} \ind{|\vbY_{-j}|>1}}{\sum_{i\in\Zset} 
      \ind{|\vbY_i|>1}} \right] \\
    & =  \esp \left[ \frac{ \sum_{j\in\Zset}\ind{\vbY_{-\infty,-j-1}^* \leq 1} \ind{|\vbY_{-j}|>1}}{\sum_{i\in\Zset} 
      \ind{|\vbY_i|>1}} \right] \\
    & =  \esp \left[ \frac{ \ind{\vbY^* > 1} }{\sum_{i\in\Zset} \ind{|\vbY_i|>1}} \right] 
      =  \esp \left[ \frac{1}{\sum_{i\in\Zset} \ind{|\vbY_i|>1}} \right] \; .
  \end{align*}
  The expressions (\ref{eq:candidate-infargmax}) and (\ref{eq:EI-debicky-hashorva}) were obtained by
  \cite{debicki:hashorva:2017} for max-stable processes, the expression~(\ref{eq:vartheta-diff}) is
  due to \cite[Remark~4.7]{basrak:segers:2009}   (\ref{eq:candidate-hashorva}) is due to Enkelejd Hashorva
  (personal communication). 
\end{example}

\begin{example}  
  In the case $d=1$, for $x\in\Rset$ define
  $x^{\langle\alpha\rangle}=\max(x,0)^\alpha-\max(-x,0)^\alpha =x|x|^{\alpha-1}$. Applying
  \Cref{lem:general-identity-forward} to the function
  $H(\vby)=\sum_{j\in\Zset} y_j^{\langle\alpha\rangle}$ which trivially satisfies its assumptions
  yields
  \begin{align}
    \vartheta \esp \left[ \sum_{j\in\Zset} Q_j^{\langle\alpha\rangle} \right]  
    = \esp[\Theta_0^{\langle\alpha\rangle}]  = \esp[\Theta_0]  \; .  \label{eq:DH95-3.13}
  \end{align}
  This identity was indirectly obtained in the proof of \cite[Theorem~3.2]{davis:hsing:1995} in the
  case $\alpha\in[1,2]$ by identification of two terms.  Recall that $\esp[\Theta_0]$ is the
  skewness of the tail of the marginal distribution of the original time series $\sequ{X}$.
\end{example}

In \cite[Theorem 4.5]{basrak:planinic:soulier:2016}, the condition
$\esp[\norm[1]{\vbQ}^\alpha]<\infty$ is used in order to establish functional convergence of the
partial sum process of a weakly dependent stationary regularly varying time-series.  For
$\alpha\leq1$, the concavity of the function $x\mapsto x^\alpha$ and (\ref{eq:sumQtheta-egal1})
implies that $\esp[\norm[1]{\vbQ}^\alpha]<\infty$. For $\alpha>1$, the latter integrability
condition does not always hold and we now obtain a sufficient condition in terms of the forward
spectral tail process using \Cref{lem:general-identity-forward}.
\begin{lemma}
  \label{coro:equivalences}
  Assume that $\pr(\lim_{|t|\to\infty} |\vbY_t|=0)=1$. Then
  \begin{align}
    \vartheta  \esp\left[ \left(\sum_{j\in\Zset} |\vbQ_j|\right)^\alpha\right] 
    & = \esp \left[ \left(\sum_{j\in\Zset} \norm{\vbTheta_j}\right)^{\alpha-1} \right] \; .   
    \label{eq:equivalence3}
  \end{align} 
  These quantities are always finite if $\alpha\leq 1$ and are simultaneously finite or infinite if
  $\alpha>1$. Moreover, the following conditions are equivalent:
  \begin{align}
      \label{eq:summability-Q}
    \esp\left[ \left(\sum_{j\in\Zset} |\vbQ_j|\right)^\alpha\right] <\infty \; , \\
    \esp \left[ \left(\sum_{j=0}^\infty \norm{\vbTheta_j}\right)^{\alpha-1} \right]   < \infty \; .
      \label{eq:summability-theta-forward}
  \end{align}  
\end{lemma}

\begin{proof}
  Applying \Cref{prop:loi-Q} to the shift invariant $\alpha$-homogeneous function $\vby\mapsto\norm[1]{\vby}^{\alpha}$
  and the time change formula~(\ref{eq:time-shift-homo-alpha}) to the $\alpha$-homogeneous function
  $\vby\mapsto|\vby_0|\norm[1]{\vby}^{\alpha-1}\ind{I(\vby)=j}$ we obtain
  \begin{align*}
    \vartheta  \esp\left[ \norm[1]{\vbQ}^\alpha\right] 
    & =     \esp \left[ \norm[1]{\vbTheta}^{\alpha}  \ind{I(\vbTheta)=0}\right]   \\
    & = \sum_{j\in\Zset} \esp \left[ |\vbTheta_{-j}|  \norm[1]{\vbTheta}^{\alpha-1}  \ind{I(\shift^j\vbTheta)=j} \ind{|\vbTheta_{-j}|\ne0}\right]  \\
    & = \sum_{j\in\Zset} \esp \left[ |\vbTheta_0| \norm[1]{\vbTheta}^{\alpha-1}  \ind{I(\vbTheta)=j} \right]  \\
    & = \esp \left[ \norm[1]{\vbTheta}^{\alpha-1} \sum_{j\in\Zset}  \ind{I(\vbTheta)=j} \right]  
      = \esp[\norm[1]{\vbTheta}^{\alpha-1}] \; . 
  \end{align*}    
  Hence, (\ref{eq:equivalence3}) holds and in particular (\ref{eq:summability-Q}) implies (\ref{eq:summability-theta-forward}). Moreover, since $\norm[1]{\vbTheta}\geq 1$,
  (\ref{eq:summability-Q}) always holds when $\alpha\leq 1$ as already noted in the discussion preceding this lemma. Conversely, if (\ref{eq:summability-theta-forward}) holds and $\alpha>1$,
  applying \Cref{lem:general-identity-forward} to the function $H(\vbx) = \norm[1]{\vbx}^\alpha$ yields~(\ref{eq:summability-Q}). Just note that the condition~\ref{item:integrabilite-diff} is implied by the bound~(\ref{eq:easy_bound}) below and the fact that $\norm[1]{\bsTheta_{0,\infty}}=1+\norm[1]{\bsTheta_{1,\infty}}$. 
\end{proof}

\begin{corollary}
  \label{lem:identites-again}
  Assume that $\pr(\lim_{|t|\to\infty} |\vbY_t|=0)=1$ and that~(\ref{eq:summability-Q}) or
  (\ref{eq:summability-theta-forward}) holds.  Let $H$ be a non negative, shift invariant,
  $\alpha$-homogenous measurable function on~$\setE{d}$ such that
  \begin{align}
    \label{eq:H-lipshitz}
    |H(\vbx) - H(\vby)| \leq C\norm[1]{\vbx-\vby}  \; , 
  \end{align}
  for some constant $C>0$ and all $\vbx,\vby\in \setE{d} $.  Then
  $\esp[|H^\alpha(\vbTheta_{0,\infty})-H^\alpha(\vbTheta_{1,\infty})|] <\infty$ and
  \begin{align}
    \vartheta \esp[H^\alpha(\vbQ)] 
    = \esp[H^\alpha(\vbTheta_{0,\infty})-H^\alpha(\vbTheta_{1,\infty})] \; .     \label{eq:newidentity-2}
  \end{align}
\end{corollary}

\begin{proof}
  Apply \Cref{lem:general-identity-forward} to the function $H^\alpha$ which satisfies its
  assumptions in view of~(\ref{eq:H-lipshitz}) and the following bounds: for all $a,b\geq 0$,
  \begin{align}\label{eq:easy_bound}
    |a^\alpha-b^\alpha| \leq
    \begin{cases}
      |a-b|^{\alpha}  & \mbox{  if $\alpha\leq 1$ } \; ,  \\
      \alpha |a-b| (a\vee b)^{\alpha -1} & \mbox{ if $\alpha>1$ } \; .
    \end{cases}
  \end{align}
\end{proof}

\begin{example}
  \label{xmpl:cluster-ruin}
 In the case $d=1$, considering the functionals $\vby\mapsto \left(\sum_{j\in\Zset} y_j\right)_+$ and
  $\vby \mapsto \left(\sup_{k\in\Zset} \sum_{j\leq k} y_j\right)_+$, under
  condition~(\ref{eq:summability-Q}), we obtain
  \begin{align}
    \vartheta \esp \left[ \left(\sum_{j\in\Zset} Q_j\right)_+^\alpha \right] 
    &  = \esp \left[ \left(\sum_{j=0}^\infty  \Theta_j\right)_+^\alpha - \left(\sum_{j=1}^\infty  \Theta_j\right)_+^\alpha\right]  \; , 
      \label{eq:cluster-index0} \\
    \vartheta \esp \left[ \left(\sup_{k\in\Zset} \sum_{j\leq k} Q_j\right)_+^\alpha \right] 
    &  = \esp \left[ \left(\sup_{k\geq0} \sum_{j=0}^k  \Theta_j\right)_+^\alpha - \left(\sup_{k\geq1} \sum_{j=1}^k  \Theta_j\right)_+^\alpha\right]  \; .
    \label{eq:ruin-index} 
  \end{align}
  The quantity in the left hand side of~(\ref{eq:cluster-index0}) appeared in
  \cite{davis:hsing:1995} in relation to the skewness of the limiting stable law of the partial sums
  when $0 < \alpha <2$.  The right hand side appears in \cite{mikosch:wintenberger:2014} under the
  name cluster index and was also related to the limiting stable distribution of the partial sums in
  \cite{mikosch:wintenberger:2016} but for $\alpha\ne1$. The quantity in the right hand side
  of~(\ref{eq:ruin-index}) appeares in \cite{mikosch:wintenberger:2016} in relation to ruin
  probabilities. 
\end{example}

We conclude this section by an example in the case $\alpha=1$, related to the location parameter of
the limiting $1$-stable law of the partial sum process of a weakly dependent regularly varying time
series with tail index~$1$. An implicit and very involved expression was given in
\cite[Theorem~3.2]{davis:hsing:1995}.  An explicit expression is given in \cite[Theorem~4.5 and
Remark~4.8]{basrak:planinic:soulier:2016} under the condition (\ref{eq:logsum-Q}) below.  The
following lemma shows that this additional integrability condition is very light and moreover allows
to express the location parameter in terms of the (forward) spectral tail process.

\begin{lemma}
  Assume that $\alpha=1$ and $\PP(\lim_{|j|\to\infty} \norm{\vbY_j}=0)=1$.  The following conditions
  are equivalent:
    \begin{align}
      & \sum_{j\in\Zset} \esp\left[|\vbQ_j|\log\left(|\vbQ_j|^{-1}\norm[1]{\vbQ}\right) \right] < \infty \; , \label{eq:logsum-Q} \\
      & \esp \left[ \log\left( \sum_{j=0}^\infty \norm{\vbTheta_j} \right) \right] < \infty  \; . \label{eq:logsum-forward} 
    \end{align}
  If either condition holds, then 
  \begin{align}
     \label{eq:thetalogtheta-summable}
    \PP\left(\sum_{j\in\Zset} |\vbTheta_j| |\log(|\vbTheta_j|)| < \infty \right) = 1 \; .
  \end{align}
  If moreover $d=1$, then
    \label{eq:expression-logsum}
    \begin{align}
      \vartheta \EE [S_{\vbQ}\log(|S_{\vbQ}|)] -       \vartheta \sum_{j\in\ZZ} \EE[Q_j \log(|Q_j|)] 
      &  = \esp \left[ S_0\log(S_0) - S_1\log(S_1)\right] \; ,  \label{expressionlgosum-forward}
    \end{align}
    with $S_{\vbQ} = \sum_{i\in\Zset} Q_i$ and $S_i = \sum_{j=i}^\infty \Theta_i$, $i=0,1$ and all the
    expectations in~(\ref{expressionlgosum-forward}) are well defined and finite.
\end{lemma}

\begin{proof}
  Note that $|\vbx_j|^{-1}\norm[1]{\vbx}\geq1$ for all $\vbx\in\ell_0$ and $j\in\Zset$. Applying
  \Cref{prop:loi-Q} to the non negative shift invariant functional 
  \begin{align*}
    \vbx \mapsto \sum_{j\in\ZZ} |\vbx_j| \log\left(|\vbx_j|^{-1} \norm[1]{\vbx}\right)
  \end{align*}
  with the convention $|\vbx_j| \log\left(|\vbx_j|^{-1} \norm[1]{\vbx}\right)=0$ when $|\vbx_j|=0$,
  and the time change formula (\ref{eq:time-shift-Theta}), we obtain 
  \begin{align}
    \vartheta\sum_{j\in\ZZ} \EE\left[|\vbQ_j| \log\left(|\vbQ_j|^{-1} \norm[1]{\vbQ}\right) \right] 
    & = \sum_{j\in\Zset} \EE\left[|\vbTheta_j| \log\left(|\vbTheta_j|^{-1} \norm[1]{\vbTheta}\right)\ind{I(\vbTheta)=0}\right]  
      \label{eq:intermediaire-log1} \\
    &  = \sum_{j\in\Zset} \EE\left[\log\left( \norm[1]{\vbTheta} \right)\ind{I(B^j\vbTheta)=0}\right] \nonumber  \\
    & =    \EE\left[ \log \left(\norm[1]{\vbTheta} \right) \right]    \;    . \label{eq:intermediaire-log2}
  \end{align}  
  This proves that all these terms are simultaneously finite or infinite and
  thus~(\ref{eq:logsum-Q}) implies (\ref{eq:logsum-forward}).

  Conversely, assume that (\ref{eq:logsum-forward}) holds.  By the time change formula, we have for
  $j\geq1$,
  \begin{align*}
    0 \leq     \EE \left[ \log(\norm[1]{\vbTheta_{-j,\infty}})  - \log(\norm[1]{\vbTheta_{-j+1,\infty}} ) \right] 
    & = \EE[ \norm{\vbTheta_j} \{\log(|\vbTheta_j|^{-1}\norm[1]{\vbTheta_{0,\infty}}) - \log(|\vbTheta_j|^{-1}\norm[1]{\vbTheta_{1,\infty}})\}] \\
    & = \EE[ \norm{\vbTheta_j} \{\log(\norm[1]{\vbTheta_{0,\infty}}) - \log(\norm[1]{\vbTheta_{1,\infty}})\}] \; .
  \end{align*}
  The quantity inside the expectation in right hand side is understood to be $0$ if
  $\norm{\vbTheta_j}=0$. Also, $\norm{\vbTheta_j}>0$ implies that $\norm[1]{\vbTheta_{1,\infty}}>0$.
  Note now that if $y\geq x>0$, then $0\leq\log(y) - \log(x) \leq (y-x)/x$. Since
  $\norm[1]{\vbTheta_{0,\infty}}-\norm[1]{\vbTheta_{1,\infty}}=1$, this yields
  \begin{align*}
    0 \leq \EE \left[ \norm{\vbTheta_j} \{\log(\norm[1]{\vbTheta_{0,\infty}}) - \log(\norm[1]{\vbTheta_{1,\infty}})\}\right] 
    \leq \EE \left[\norm[1]{\vbTheta_{1,\infty}}^{-1} \norm{\vbTheta_j} \right] \; .
  \end{align*}
  Since by assumption $0\leq\EE[\log(\norm[1]{\vbTheta_{0,\infty}})]<\infty$, summing over $j$
  yields, for $n\geq1$,
  \begin{align*}
    0 \leq \EE \left[ \log(\norm[1]{\vbTheta_{-n,\infty}})  \right] \leq \EE[\log(\norm[1]{\vbTheta_{0,\infty}})] 
    + \sum_{j=1}^n \EE \left[ \frac{\norm{\vbTheta_j}}{\norm[1]{\vbTheta_{1,\infty}}}\right] 
    \leq \EE[\log(\norm[1]{\vbTheta_{0,\infty}})] + 1 < \infty \; .
  \end{align*}
  By monotone convergence, this proves that (\ref{eq:logsum-forward}) implies
  $\EE[\log(\norm[1]{\vbTheta})]<\infty$ which
  by~(\ref{eq:intermediaire-log2}) implies~(\ref{eq:logsum-Q}).

  As a first consequence, if~(\ref{eq:logsum-Q}) holds, then the identity
  (\ref{eq:intermediaire-log2}) and \Cref{coro:nul-infargmax} imply that
  \begin{align*}
    0 \leq \sum_{j\in\Zset} |\vbTheta_j| \log(|\vbTheta_j|^{-1}\norm[1]{\vbTheta}) < \infty \; ,  \ \ \mbox{ a.s.}
  \end{align*} 
  and this in turn implies~(\ref{eq:thetalogtheta-summable}). 

  To prove the last statement, define the function $S$ on $\ell_1$ by
  $S(\vbx) = \sum_{j\in\Zset} x_j$.  Some easy calculus (cf. \Cref{sec:properties-slogs}) yields the
  following properties: for all $\vbx,\vby\in\ell_1$, such that $|S(\vbx)-S(\vby)|\leq 1$, then
  \begin{align}
  \label{eq:borne-diff-slogs}
    |S(\vbx)\log(|S(\vbx)|) - S(\vby)\log(|S(\vby)|)| & \leq 2 + \log_+(\norm[1]{\vbx}\vee\norm[1]{\vby})  \; .
  \end{align}
   Define a shift invariant, $1$-homogeneous function $H$ on the subset of $\ell_1$ of sequences such that
  $\sum_{j\in\ZZ} |x_j\log(|x_j|)|<\infty$ by
  \begin{align*}
    H(\vbx) = S(\vbx) \log(|S(\vbx)|) - \sum_{j\in\Zset} x_j\log(|x_j|) \; .
  \end{align*}
  By~(\ref{eq:thetalogtheta-summable}), we know that $\vbTheta$ is in this set and since
  $|\vbTheta_0|=1$, we have 
  \begin{align*}
    H(\vbTheta_{0,\infty})-H(\vbTheta_{1,\infty}) =
    S(\vbTheta_{0,\infty})\log(|S(\vbTheta_{0,\infty})|)-S(\vbTheta_{1,\infty})\log(|S(\vbTheta_{1,\infty})|)  
  \end{align*}
  Thus, (\ref{eq:logsum-forward}) and (\ref{eq:borne-diff-slogs}) yield
  \begin{align*}
    \esp[  |H(\vbTheta_{0,\infty}) - H(\vbTheta_{1,\infty}) | ] \leq 2 + \esp[\log(\norm[1]{\vbTheta_{0,\infty}})] < \infty  \; .
  \end{align*}
  Thus condition~\ref{item:integrabilite-diff} of \Cref{lem:general-identity-forward} holds.
  Conditions \ref{item:to0} and \ref{item:toH} trivially hold under
  (\ref{eq:thetalogtheta-summable}).  Thus we can apply
  \Cref{lem:general-identity-forward} to obtain~(\ref{eq:borne-diff-slogs}).  
 \end{proof}
\subsection{Cluster indices}

Let $H$ be a measurable shift invariant functional, homogeneous with degree~1, and continuouss on
{the injection of $(\Rset^d)^k$ into $\ell_0$ for every $k\geq1$}. 
Then, we obtain by (\ref{eq:reg_var}), $1$-homogeneity of $H$ and continuous mapping that
\begin{align*}
  \lim_{x\to\infty} \frac{\pr(H(\bsX_{1,k})>x)}{\pr(\norm{\bsX_0}>x)} 
  & = \lim_{x\to\infty} \frac{\pr(H(x^{-1}\bsX_{1,k})>1)}{\pr(\norm{\bsX_0}>x)} \\
  & = \nu_{1,k}(\{\bsx\in (\Rset^d)^k: H(\bsx)>1\})=\nu(\{\bsy \in(\Rset^d)^\Zset: H(\bsy_{1,k})>1\}) \; .
\end{align*} 
Let the limit on the left hand side or the expression in the right hand side be denoted by $ b_k(H)$
For $d=1$ and $H(\vbx)=\sum_{j\in\Zset} x_j$, the quantity $\lim_{k\to\infty}k^{-1}b_k(H)$ was
called a cluster index of the time series $\{X_j,j\in\Zset\}$ by \cite{mikosch:wintenberger:2014}.
We extend the notion of cluster index to a large class of functionals for which the limit
$\lim_{k\to\infty}k^{-1}b_k(H)$ exists.
\begin{lemma}
  \label{lem:cluster-index-general}
  Assume that $\pr(\lim_{|t|\to\infty} |\vbY_t|=0)=1$.  Let $H$ be a shift invariant,
  $1$-homogeneous functional, {continuous on the injection of $(\Rset^d)^k$ into $\ell_0$ for
    every $k\geq1$} and such that $|H(\vby)|\leq C\norm[1]{\vby}$ for a constant $C>0$ and all
  $\vby\in(\Rset^d)^\Zset$. Then
  \begin{align}
    b_{k+1}(H) - b_k(H) = \esp[H_+^\alpha(\bsTheta_{0,k})-H_+^\alpha(\bsTheta_{1,k})]  \; .  \label{eq:cluster-k}
  \end{align}
  Assume moreover that $|H(\vbx)-H(\vby)| \leq C\norm[1]{\vbx-\vby}$ for all
    $\vbx,\vby\in(\Rset^d)^\Zset$.  If $\alpha>1$ assume in addition that
  \begin{align}\label{eq:sumThetaalpha-1}
    \esp \left[ \left(\sum_{j=0}^\infty \norm{\vbTheta_j}\right)^{\alpha-1} \right] < \infty \; .
  \end{align}
  Then,
  \begin{align}
    \lim_{k\to\infty} \frac{b_k(H)}{k} = \esp [ H_+^\alpha(\vbTheta_{0,\infty}) - H_+^\alpha(\vbTheta_{1,\infty})] 
    = \vartheta \esp[H_+^\alpha(\vbQ)]\; .  \label{eq:cluster-index}
  \end{align}
\end{lemma}

\begin{proof}
  Without loss of generality, we assume that $H$ is non-negative. By definition of the tail measure
  and by stationarity, we have
  \begin{align*}
    b_k(H) & = \nu(\{H(\bsy_{1,k})>1\}) \; ,  \ \             b_{k+1}(H)  = \nu(\{H(\bsy_{0,k})>1\}) \; .
    \end{align*}
    Note that $\ind{H(\bsy_{0,k})>1} = \ind{H(\bsy_{1,k})>1}$ if $\bsy_0=0$.  Therefore, we can
    apply~(\ref{eq:polaire-tail-measure-0}) and obtain
    \begin{align*}
      b_{k+1}(H)-b_k(H) 
      & = \int_{\setE{d}} \left( \ind{H(\bsy_{0,k})>1} - \ind{H(\bsy_{1,k})>1} \right) \ind{\bsy_0\ne0} \nu(\rmd\bsy) \\
      & = \int_0^\infty \esp \left[ \left( \ind{rH(\bsTheta_{0,k})>1} -
        \ind{rH(\bsTheta_{1,k})>1} \right)  \right] \alpha r^{-\alpha-1} \rmd r  \\
      & = \esp[H^\alpha(\bsTheta_{0,k})- H^\alpha(\bsTheta_{1,k})] \; .
    \end{align*}
    We must now prove that the last expression has a limit when $k$ tends to infinity. By
    \eqref{eq:summability-theta} (if $\alpha\leq 1$) and \eqref{eq:sumThetaalpha-1} (if $\alpha>1$),
    $\sum_{j=0}^\infty \norm{\vbTheta_j}<\infty$ almost surely, so the assumption on $H$ implies
    that $H(\vbTheta_{0,k})$ converges almost surely to $H(\vbTheta_{0,\infty})$ which is well
    defined. Moreover, since $|\vbTheta_0|=1$, we obtain
    \begin{align*}
      \left| H^\alpha(\bsTheta_{0,k}) - H^\alpha(\bsTheta_{1,k}) \right|  
      \leq (\alpha\vee1) C^\alpha  \left(\sum_{j=0}^\infty \norm{\vbTheta_j}\right)^{(\alpha-1)_+} \; .
    \end{align*}
    Thus, under assumption~(\ref{eq:sumThetaalpha-1}), the limit~(\ref{eq:cluster-index}) holds by
    dominated convergence.
  \end{proof}

\begin{remark}
  As noted in \Cref{xmpl:indicator-is-important}, the fact that we integrate a function which
  vanishes when $\bsy_0=0$ is essential.  The identity (\ref{eq:cluster-k}) was obtained in
  \cite[Lemma~3.1]{mikosch:wintenberger:2014} by means of a rather lengthy proof which made repeated
  use of the definition of the spectral tail process and the time change formula.
\end{remark}

\begin{example}
  We pursue \Cref{xmpl:cluster-ruin}.  If $\pr(\lim_{|t|\to\infty}|Y_t| = 0)=1$ and
    \eqref{eq:sumThetaalpha-1} hold, then we can apply \Cref{lem:cluster-index-general} to the
    functionals of \Cref{xmpl:cluster-ruin} and obtain
  \begin{align}
  \label{eq:clusterindex-sum}
    \lim_{k\to\infty} \lim_{x\to\infty} \frac{\pr(X_1+\cdots+X_k>x)}{k\pr(|X_0|>x)} 
    &       = \esp \left[ \left(\sum_{j=0}^\infty  \Theta_j\right)_+^\alpha - \left(\sum_{j=1}^\infty  \Theta_j\right)_+^\alpha\right]  \; , \\
    \lim_{k\to\infty} \lim_{x\to\infty} \frac{\pr(\sup_{1\leq j \leq k} (X_1+\cdots+X_j)>x)}{k\pr(|X_0|>x)} 
    &  = \esp \left[ \left(\sup_{k\geq0} \sum_{j=0}^k  \Theta_j\right)_+^\alpha - \left(\sup_{k\geq1} \sum_{j=1}^k  \Theta_j\right)_+^\alpha\right]  \; .
  \end{align}
  The identity (\ref{eq:clusterindex-sum}) was proved for geometrically ergodic Markov chains by
  \cite[Theorem~3.2]{mikosch:wintenberger:2014}. 
\end{example}

\section{Convergence  of clusters}
\label{sec:emp-proc-clusters}
The quantities studied in \Cref{sec:tailprocesstozero} appear as limits of so-called cluster
functionals. To be precise, a cluster is a vector $\vbX_{1,r_n}=(\vbX_1,\dots,\vbX_{r_n})$ where
$\{r_n\}$ is a non decreasing sequence of integers such that $\lim_{n\to\infty} r_n=\infty$. The
vector $\vbX_{1,r_n}$ can be embedded in $(\Rset^d)^\Zset$ and cluster functional may be simply
defined as measurable function $H$ on $(\Rset^d)^\Zset$. Limiting theory for regularly varying time
series relies fundamentally on the convergence of functionals of renormalized clusters, that is the
convergence of
\begin{align}
  \frac{  \esp[H(c_n^{-1} \vbX_{1,r_n})]}{r_n\pr(\norm{\vbX_0}>c_n)} \; ,   \label{eq:espcluster}
\end{align}
where $\{c_n\}$ an increasing sequence such that $\lim_{n\to\infty} c_n = \infty$ and
\begin{align}
  \lim_{n\to\infty} r_n\pr(\norm{\vbX_0}>c_n) = 0 \; .  \label{eq:cn-domine-rn}
\end{align}
The convergence of the quantity in~(\ref{eq:espcluster}) has been established under various
conditions on the functions $H$, in particular some form of shift invariance, and more essentially
under the following so-called anticlustering condition, originally introduced in
\cite[Condition~(2.8)]{davis:hsing:1995}: for all $u>0$,
\begin{align}
  \lim_{m\to\infty} \limsup_{n\to\infty}  \pr\left( \max_{m\leq |i| \leq r_n} \norm{\bsX_i} > c_n u \mid \norm{\bsX_0}>c_nu\right) = 0 \; .
  \label{eq:anticlustering}
\end{align}
It is proved in \cite[Proposition~4.1]{basrak:segers:2009} that  (\ref{eq:anticlustering})
implies~(\ref{eq:Y_to_zero}), \ie\
$\pr \left( \lim_{|k|\to\infty} \norm{\bsY_k} = 0 \right) = 1$.  In full generality, condition
(\ref{eq:anticlustering}) cannot bring more information on the tail process since it is proved in
\cite{debicki:hashorva:2017} that for max-stable stationary processes with Fr\'echet marginal,
(\ref{eq:anticlustering}) and (\ref{eq:Y_to_zero}) are equivalent; see \Cref{sec:maxstableprocesses}.

In order to give a rigorous meaning to the convergence of clusters, following
\cite{basrak:planinic:soulier:2016}, we consider clusters as element of the space $\lo$ of shift
equivalent sequences. More precisely, we say that $\vbx\sim\vby$ if there exists $j\in\Zset$ such
that $\shift^j\vbx=\vby$. The space $\lo$ is the space of equivalence classes.  It is readily
checked that the space $\lo$ endowed with the metric $\dtilde$ defined by
\begin{align*}
  \dtilde (\tilde\vbx,\tilde\vby) = \inf_{\vbx\in\tilde\vbx,\vby\in\tilde\vby} \norm[\infty]{\vbx-\vby}
\end{align*}
is a complete separable metric space. See \cite[Lemma~2.1]{basrak:planinic:soulier:2016}.

Define the measure $\nu_{n,r_n}$ on $\lo$ by 
\begin{align*}
  \nu_{n,r_n} = \frac{\pr \left( c_n^{-1} \vbX_{1,r_n} \in \cdot \right)} {r_n\pr(\norm{\vbX_0}>c_n)}  \; .
\end{align*}
The convergence of the quantity in~(\ref{eq:espcluster}) can now be related to the convergence of
the measure $\nu_{n,r_n}$ on the space $\loo$ in the following sense.

Let $\mcm_0$ be the set of boundedly finite Borel measures on $\loo$, that is Borel measures $\mu$
such that $\mu(A)<\infty$ for all Borel sets $A\subset\lo$ which are bounded away from
$\tilde\bszero$ \ie\ for which there exists $\epsilon>0$ such that $\tilde\vbx\in A$ implies that
$\tilde\vbx^*>\epsilon$.

Following \cite{hult:lindskog:2006} or \cite[Chapter~4]{kallenberg:2017}, we say that a sequence of
measures $\mu_n\in\mcm_0$ converge to $\mu$ in $\mcm_0$ if
$\lim_{n\to\infty}\mu_n(f)=\mu(f)$ for all bounded continuous functions $f$ on $\loo$ with support
bounded away from zero. As shown in \cite[Lemma~4.1]{kallenberg:2017}, the class of test functions
can be restricted to Lipschitz continuous functions. Let $\nu^*$ be the measure defined on $\lzero$
by
\begin{align*}
  \nu^*(H) = \vartheta \int_0^\infty \esp[H(r\vbQ)] \alpha r^{-\alpha-1} \rmd r  \; ,
\end{align*}
for non negative measurable functions $H$ defined on $\lzero$.  Since $\vbQ^*=1$, the measure
$\nu^*$ is boundedly finite on $\lzero$. By \Cref{prop:loi-Q}, if $H$ is shift invariant then
$\nu^*(H)$ has the alternative expression
\begin{align*}
  \nu^*(H)=\int_0^\infty \esp[H(r\vbTheta)\ind{I(\vbTheta)=0}]\alpha r^{-\alpha-1} \rmd r \; .
\end{align*}
By a slight abuse of notation, in the following we consider $\nu^*$ as a measure on $\lo$. It is
proved in \cite[Lemma~3.3]{basrak:planinic:soulier:2016} that if $\vbX$ is a stationary regularly
varying time series with tail measure $\nu$ and which satisfies condition (\ref{eq:anticlustering}),
then $\nu_{n,r_n} \to \nu^*$ in $\mcm_0$. This convergence implies that for all bounded measurable
shift invariant functions $H$ on $\lzero$ (which can be identified with functions on $\lo$) with
support bounded away from zero and almost surely continuous \wrt\ $\nu^*$,
\begin{align*}
  \lim_{n\to\infty} \frac{\esp[H(c_n^{-1} \vbX_{1,r_n})]}{r_n\pr(\norm{\vbX_0}>c_n)} = \lim_{n\to\infty} \nu_{n,r_n}(H) = \nu^*(H) \; . 
\end{align*}
This approach unifies and extends similar results in \cite{basrak:segers:2009} and
\cite{mikosch:wintenberger:2014,mikosch:wintenberger:2016}. The extension of the above convergence
to unbounded functions or functions whose support is not bounded away from $\bszero$ can be obtained
by usual uniform integrability arguments.

The previous results were proved under the anticlustering condition (\ref{eq:anticlustering}). In
particular, as already mentioned, this always implies that the tail process tends to zero at
infinity. However, whereas for most time series models (such as linear models or solutions to
stochastic recurrence equations), it is relatively easily checked that the tail process tends to
zero, proving condition~(\ref{eq:anticlustering}) is relatively hard and may require very stringent
conditions.

We next show that the anticlustering condition (\ref{eq:anticlustering}) is actually not
essential. Recall that $\pr \left( \lim_{|k|\to\infty} \norm{\bsY_k} = 0 \right) = 1$ is equivalent
to tail measure $\nu$ being supported on $\lzero$.

\begin{lemma}
  \label{lem:anticlustering-useless}
  Let $\vbX$ be a regularly varying time series with tail measure $\nu$ supported on $\lzero$.  Then
  for every sequence $\{c_n\}$ such that $\lim_{n\to\infty} c_n=\infty$, there exists a non
  decreasing sequence of integers $\{r_n\}$ such that (\ref{eq:cn-domine-rn}) holds and
  $\nu_{n,r_n} \to \nu^*$ in $\mcm_0$.
\end{lemma}

\begin{proof}
  For each integer $m\geq1$ and for every non negative shift invariant function $H$ on $\lzero$ such
  that $H(\vbx)=0$ if $\vbx^*\leq\epsilon$ and continuous with respect to the distribution of
  $\vbY$, we have, by regular variation,
  \begin{align*}
    \lim_{n\to\infty} \frac{\esp[H(c_n^{-1}\vbX_{1,m})]}{m\pr(\norm{\vbX_0}>c_n)} = 
    \frac{\epsilon^{-\alpha}}m \sum_{j=1}^m \esp[H(\epsilon\vbY_{1-j,m-j})\ind{\vbY_{1-j,-1}^*\leq1}] \; .
  \end{align*}
  The limit is independent of $\epsilon$ and therefore defines a boundedly finite measure $\nu_m^*$
  on $\lo\setminus\{\bszero\}$.  By \cite[Lemma~4.1]{kallenberg:2017}, it suffices to prove that for
  all bounded Lipshitz continuous (with respect to the uniform norm) functions $H$ on $\loo$ with
  support bounded away from zero, we have
  \begin{align*}
    \lim_{m\to\infty} \nu_m^*(H)  &  = \nu^*(H) \; .
  \end{align*}
  The class of test functions can be further restricted to functions which depend only on coordinate
  greater than some $\eta>0$. Indeed, let $T_\eta$ be the operator on $\lzero$ which puts all
  coordinates no greater than $\eta$ to zero:
  \begin{align*}
    T_\eta(\vbx) = (\vbx_j\ind{\norm{\vbx_j}>\eta})_{j\in\Zset} \; ,
  \end{align*}
  and identify $T_\eta$ to an operator on $\loo$ in an obvious way. Then if $H$ is Lipshitz
  continuous, there exists a constant $C$ (depending only on $H$) such that for all $\vbx\in\lzero\setminus\{\bszero\}$,
  \begin{align*} 
    |H(\vbx)-H\circ T_\eta(\vbx)|\leq C \eta  \; . 
  \end{align*}
  Moreover, $H\circ T_\eta$ is almost surely continuous with respect to the distribution of $\vbY$
  since $\norm{\vbY_0}$ has a continuous distribution and is independent of $\vbTheta$ so
  $\pr(\exists j \in \Zset, \norm{\vbY_j}=\eta)=0$ for all $\eta>0$. Consider now a function $H$ with
  support bounded away from $\bszero$, which depends only on coordinates greater than $\eta$
  (that is such that $H = H\circ T_\eta$) and almost surely continuous with respect to the
  distribution of $\vbY$. Then we can write 
  \begin{align*}
    \nu_m^*(H) = 
    \frac{\epsilon^{-\alpha}}m \sum_{j=1}^m \esp[H(\epsilon\vbY_{1-j,m-j})\ind{\vbY_{1-j,-1}^*\leq1}] 
    = \epsilon^{-\alpha} \int_0^1 g_m(t) \rmd t 
  \end{align*} 
  with
  $g_m(t) = \esp[H(\epsilon\vbY_{1-\lceil mt\rceil,m-\lceil mt\rceil}) \ind{\vbY_{1-\lceil
      mt\rceil,-1}^*\leq1}]$
  (where $\lceil x\rceil$ denotes the smallest integer larger than or equal to the real number $x$).
  Since $H$ is shift invariant, depends only on coordinates greater than $\eta$ and
  $\pr(\lim_{|j|\to\infty} \norm{\vbY_j}=0)=1$, for every $t\in(0,1)$, it holds that
  \begin{align*}
    H(\epsilon\vbY_{1-\lceil mt\rceil,m-\lceil mt\rceil})\ind{\vbY_{1-\lceil mt\rceil,-1}^*\leq1} = H(\epsilon\vbY)\ind{\vbY_{1-\lceil mt\rceil,-1}^*\leq1}
  \end{align*}
  for large enough $m$. Also,
  $\lim_{m\to\infty} H(\epsilon\vbY)\ind{\vbY_{1-\lceil mt\rceil,-1}^*\leq1} =
  H(\epsilon\vbY)\ind{\vbY_{-\infty,-1}^*\leq1}$. Since $H$ is bounded, we obtain by dominated convergence that
  \begin{align*}
    \lim_{m\to\infty} g_m(t) = \esp[H(\epsilon\vbY)\ind{\vbY_{-\infty,-1}^*\leq1}] \; , 
  \end{align*}
  for all $t\in(0,1)$. The functions $g_m$ are uniformly bounded thus by dominated convergence again, we obtain
  \begin{align*}
    \lim_{m\to\infty} \nu_m^*(H) 
    & =  \epsilon^{-\alpha} \lim_{m\to\infty} \int_0^1 g_m(t)  \rmd t  
      = \epsilon^{-\alpha} \esp[H(\epsilon\vbY)\ind{\vbY_{-\infty,-1}^*\leq1}] = \nu^*(H) \; .
  \end{align*}
  This proves that $\nu_m^*$ converges to $\nu^*$ in $\mcm_0$.  Since convergence in $\mcm_0$ is
  metrizable (cf. \cite[Theorem~2.3]{hult:lindskog:2006}, there exists a sequence $r_n$ such that
  $\nu_{n,r_n}\to \nu^*$.

  As a consequence, we obtain for all $u>0$,
  \begin{align}
    \label{eq:maximumconvergence}
    \lim_{n\to\infty} \frac{\pr(\vbX_{1,r_n}^*>c_nu)}{r_n\pr(\norm{\vbX_0}>c_n)} = \vartheta u^{-\alpha} \; . 
  \end{align}
  This in turn implies that $\lim_{n\to\infty} r_n\pr(\norm{\vbX_0}>c_n)=0$. Otherwise
    $r_n\pr(\norm{\vbX_0}>c_n)\to c\in(0,\infty]$ (possibly along a subsequence) which implies that
    $\pr(\vbX_{1,r_n}^*>c_nu)\to c\vartheta u^{-\alpha}$ for all $u>0$. This is impossible since
    the latter quantity is greater than 1 for small $u$. 
\end{proof}
The convergence~(\ref{eq:maximumconvergence}) was proved under condition~(\ref{eq:anticlustering})
by \cite{basrak:segers:2009}.  Here, we have bypassed the anticlustering condition
(\ref{eq:anticlustering}).  The sequence $\{r_n\}$ is not explicitely known, but neither is it when
condition (\ref{eq:anticlustering}) is simply assumed as often happens in the literature.

We next show that the convergence $\nu_{n,r_n} \to \nu^*$ is equivalent to
\eqref{eq:maximumconvergence} and convergence in distribution of the normalized block
$(\vbX_{1,r_n}^*)^{-1}\vbX_{1,r_n}$ conditionally on $\vbX_{1,r_n}^*>c_nu$ in $\lo$ to the sequence
$\vbQ$. Note that, since the convergence takes place in the space $\lo$, by \Cref{prop:loi-Q} the
limit has the same distribution as $\vbTheta$ conditionally on $ I(\vbTheta)=0$.

\begin{lemma}
  \label{lem:clusterconvergence}
  Let $\vbX$ be a stationary regularly varying time series with tail measure $\nu$ supported on
  $\lzero$ and let $\{c_n\}$ and $\{r_n\}$ be sequences satisfying \eqref{eq:cn-domine-rn}. Then
  $\nu_{n,r_n} \to \nu^*$ in $\mcm_0$ if and only if for every $u>0$ \eqref{eq:maximumconvergence} holds and 
 \begin{align}
   \label{eq:normalizedclusterconvergence}
   \law\left((\vbX_{1,r_n}^*)^{-1}\vbX_{1,r_n} \mid
   \vbX_{1,r_n}^*>c_nu\right)  \wto \law(\vbQ)
    \end{align}
 as $n\to\infty$ in $\lo$.
\end{lemma}

\begin{proof}
  Assume first that for every $u>0$, \eqref{eq:maximumconvergence} and
  \eqref{eq:normalizedclusterconvergence} hold. It suffices to show that then for every $u>0$
\begin{align}
   \label{eq:clusterconvergence}
   \law\left((c_n u)^{-1}\vbX_{1,r_n} \mid
   \vbX_{1,r_n}^*>c_nu\right)  \wto \law(Y \cdot \vbQ)\;  
 \end{align}
 in $\lo$, where $Y$ is a Pareto distributed random variable independent of $\vbQ$, since the fact
 that \eqref{eq:maximumconvergence} and \eqref{eq:clusterconvergence} imply $\nu_{n,r_n} \to \nu^*$
 follows as in \cite[Lemma 3.2]{basrak:planinic:soulier:2016}. Fix $u>0$ and take an arbitrary
 $v\geq 1$ and a Borel subset $B$ of $\loo$ such that $\pr(\vbQ\in \partial B)=0$. Then by
 \eqref{eq:maximumconvergence}, \eqref{eq:normalizedclusterconvergence} and regular variation of
 $|\vbX_{0}|$, as $n\to \infty$
  \begin{multline*}
  \pr\left(\vbX_{1,r_n}^*>c_nuv,(\vbX_{1,r_n}^*)^{-1}\vbX_{1,r_n} \mid
   \vbX_{1,r_n}^*>c_nu \right)\\
   = \frac{\pr\left(\vbX_{1,r_n}^*>c_nuv\right)}{\pr\left(\vbX_{1,r_n}^*>c_nu \right)} \cdot \pr\left((\vbX_{1,r_n}^*)^{-1}\vbX_{1,r_n} \mid
   \vbX_{1,r_n}^*>c_nuv\right) \\
   \to v^{-\alpha} \cdot \pr(\vbQ \in B) \; .
\end{multline*}    
This implies that for every $u>0$ 
\[
\law\left((c_n u)^{-1}\vbX_{1,r_n}^*,(\vbX_{1,r_n}^*)^{-1}\vbX_{1,r_n} \mid
   \vbX_{1,r_n}^*>c_nu\right)\wto \law(Y,\vbQ)
   \]
 in $(1,\infty)\times \lo$ and \eqref{eq:clusterconvergence} now follows by an continuous mapping argument.
  
 For the converse, assume that $\nu_{n,r_n} \to \nu^*$ in $\mcm_0$. As already noted in the proof of
 Lemma \ref{lem:anticlustering-useless}, this implies that \eqref{eq:maximumconvergence} holds for
 every $u>0$. Fix an $u>0$ and take an arbitrary bounded continuous function $H$ on $\lo$. Note that
 the function $\tilde{\vby} \mapsto H((\tilde{\vby}^*)^{-1} \tilde{\vby})\ind{\tilde{\vby}^*>u}$ on
 $\loo$ is bounded, has support bounded away from $\bszero$ and is almost surely continuous with
 respect to $\nu^*$ since $\nu^*(\{\tilde{\vby}: \: \tilde{\vby}^*=u \})=0$ by the definition of
 $\nu^*$ and the fact that $\vbQ^*=1$. Now by the convergence $\nu_{n,r_n} \to \nu^*$ in $\mcm_0$
 and \eqref{eq:maximumconvergence}, as $n\to \infty$
 \begin{multline*}
   \esp\left[H((\vbX_{1,r_n}^*)^{-1}\vbX_{1,r_n}) \mid  \vbX_{1,r_n}^*>c_nu\right]
   \\
   =\frac{r_n\pr(\norm{\vbX_0}>c_n)}{\pr(\vbX_{1,r_n}^*>c_nu)}\cdot\frac{\esp\left[H((\vbX_{1,r_n}^*)^{-1}\vbX_{1,r_n})\ind{
     \vbX_{1,r_n}^*>c_nu}\right]}{r_n\pr(\norm{\vbX_0}>c_n)}\\
   \to \vartheta^{-1} u^{\alpha} \int_{\loo}H((\tilde{\vby}^*)^{-1}\tilde{\vby})\ind{\tilde{\vby}^*>u} \nu^*(\rmd\tilde{\vby})=\esp[H(\vbQ)] \; .
 \end{multline*}

\end{proof}

Assume now that $n\pr(\norm{\vbX_0}>c_n)\sim 1$ and that $\{c_n\}$ and $\{r_n\}$ satisfy
  the assumption of \Cref{lem:anticlustering-useless}.  Define $k_n=[n/r_n]$,
  $\vbX_{n,i} = c_n^{-1}(X_{(i-1)r_n+1},\dots,X_{ir_n})$, $i=1,\dots,k_n$ and the point process of
  clusters
\begin{align*}
  \PPC_n = \sum_{i=1}^{k_n} \delta_{\frac{i}{k_n},\vbX_{n,i}} \; .
\end{align*}
The point process $\PPC_n$ is a generalization introduced in \cite{basrak:planinic:soulier:2016} of
the point process of exceedences $N_n=\sum_{k=1}^n \delta_{X_k/c_n}$ and of its functional version
$N_n=\sum_{k=1}^n \delta_{i/n,X_k/c_n}$ considered in \cite{davis:hsing:1995} and
\cite{basrak:krizmanic:segers:2012}.  The convergence of these point processes is a central tool in
obtaining limit theorems for heavy tailed time series.

The convergence of $\PPC_n$ to a Poisson point process on~$[0,1]\times\loo$ with mean measure
$\mathrm{Leb} \times \nu^*$ is proved in \cite[Theorem~3.6]{basrak:planinic:soulier:2016} under the
anticlustering condition~(\ref{eq:anticlustering}) and the following mixing condition:
\begin{align}
  \esp\left[ \rme^{-\PPC_n(f)}\right ] - \prod_{i=1}^{k_n} \esp \left[ \rme^{-f(i/k_n,\vbX_{n,i})} \right] \to 0  \; ,  \label{eq:Asecond}
\end{align}
where $f$ is a continuous non negative function on $[0,1]\times\loo$ with support bounded away from
$[0,1]\times\{\tilde\bszero\}$ and $\vbX_{n,i}$ is identified to an element of $\lo$. The condition
(\ref{eq:Asecond}) has been shown in \cite[Lemma~6.5]{basrak:planinic:soulier:2016} to hold under
$\beta$-mixing and it probably also holds under $\alpha$-mixing. However, many processes of interest
are neither $\beta$- nor $\alpha$-mixing, for instance, linear processes without stringent
assumptions on the distribution of the innovation or long memory linear processes.

It must be noted that if $\nu_{n,r_n}\to\nu^*$ in $\mcm_0$, then condition~(\ref{eq:Asecond}) is a
necessary and sufficient condition for the convergence of $\PPC_n$ to a Poisson point process $\PPC$
with mean measure $\mathrm{Leb} \times \nu^*$ on $[0,1]\times\loo$. Indeed, the convergence
$\nu_{n,r_n}\to\nu^*$ implies that
\begin{align*}
  \lim_{n\to\infty}\prod_{i=1}^{k_n} \esp \left[ \rme^{-f(i/k_n,\vbX_{n,i})} \right] = \esp\left[\rme^{-\PPC(f)}\right]  \; ,  
\end{align*}
for all functions $f$ as before (cf. \cite[Proposition 3.21]{resnick:1987}) and this is also the
limit of $\esp\left[\rme^{-\PPC_n(f)}\right]$ if $\PPC_n$ converges weakly to $\PPC$. Thus, the two
quantities in (\ref{eq:Asecond}) have the same limit and their difference tends to zero.

Consider again linear processes.  Since the convergence $\PPC_n$ to $\PPC$ is known to hold for
any~$r_n$ such that $r_n\to\infty$ and $r_n/n\to 0$, (by an argument of $m$-dependent approximation,
cf. \cite[Proposition~3.8]{basrak:planinic:soulier:2016}), thus (\ref{eq:Asecond}) holds (for the
same sequence $r_n$ we get from \Cref{lem:anticlustering-useless}) even when the linear process is
not mixing.  

In view of these remarks, it is not suprising that conditions~(\ref{eq:anticlustering})
and~(\ref{eq:Asecond}) are relatively hard to check since they are nearly necessary and sufficient
conditions for the point process convergence. Unfortunately, no more easily checked sufficient
conditions (other than $\beta$-mixing) are known.

\section{Max-stable processes with a given tail measure}
\label{sec:maxstableprocesses}
In this section, we recall some connections between max-stable processes and spectral tail
processes. In particular, we provide an alternative construction based on the tail measure of
\cite[Theorem~3.2]{janssen:2017} which states that given a non negative process $\vbTheta$
satisfying the time change formula and $\pr(\Theta_0=1)=1$, there exists a stationary max-stable
process with spectral tail process $\vbTheta$. For brevity, we only consider non negative real
valued max-stable processes. The extension to the $d$-dimensional case is straightforward.  We only
consider the case $\lim_{|j|\to\infty}\Theta_j=0$ here. The general (non negative) case is
considered in \cite[Theorem~4.2]{janssen:2017}. Further generalizations in connection with the tail
measure are considered in \cite{dombry:hashorva:soulier:2017}.

We first recall some results about stationary max-stable processes with Fr\'echet marginals.  Let
$\bzeta$ be a max-stable process which admits the representation
\begin{align}
  \zeta_j = \bigvee_{i=1}^\infty P_i Z_j^{(i)} \; , \ \ j\in\Zset \; ,  \label{eq:def-zeta-Z}
\end{align}
where $\sequ{P}[i][\Nset]$ are the points of a Poisson point process on $(0,\infty)$ with mean
measure $\alpha x^{-\alpha-1}\rmd x$ and $\sequ{Z^{(i)}}$, $i\geq1$ are \iid\ copies of a non
negative process $\bsZ$ such that $\esp[Z_j^\alpha]=1$ for all $j\in\Zset$. The marginal
distributions are standard $\alpha$-Fr\'echet and the condition for stationarity is that $\bsZ$
satisfies
\begin{align*}
  \esp \left[ \bigvee_{i=s}^t \frac{Z_i^\alpha}{x_i^\alpha} \right] 
  = \esp \left[ \bigvee_{i=s}^t \frac{Z_{i-k}^\alpha}{x_i^\alpha} \right]   \; , 
\end{align*}
for all $k,s\leq t\in\Zset$ and $x_i>0$ for $i=s,\dots,t$.  The marginal distribution of $\zeta_0$
is unit Fr\'echet, the process $\bzeta$ is regularly varying and it is proved in
\cite[Section~6.2]{hashorva:2016} that the distribution of its spectral tail process $\vbTheta$ is
given, for all $h\in\Zset$ and bounded measurable functions $F$ on $(\Rset^d)^\Zset$, by
\begin{align}
  \esp[F(\vbTheta)] =  \esp[Z_{-h}^\alpha F(\shift^hZ/Z_{-h})\ind{Z_{-h}\ne0}]  \; . \label{eq:Z-theta}
\end{align}
It is also proved in \cite[Section~6.2]{hashorva:2016} that the distribution of $\bzeta$ is
characterized by its spectral tail process via the infargmax formula:
\begin{align}
  \label{eq:infargmax-formula}
  -\log \pr\left( \zeta_j\leq y_j,j\in\Zset\right) 
  & = \sum_{h\in\Zset} \frac1{y_h} \pr \left( \inf\arg\max_{j\in\Zset}\frac{\Theta_{j}}{y_{j+h}} = 0
    \right) \; , 
\end{align}
where only finitely many of the positive numbers $y_j$ are finite. 

Furthermore, \cite{debicki:hashorva:2017} proved that the process $\bzeta$ satisfies the
anticlustering condition (\ref{eq:anticlustering}) for any sequences $\{c_n\}$ and $\{r_n\}$
  such that $\lim_{n\to\infty} r_n\pr(\zeta_0>c_n)=\lim_{n\to\infty} r_n c_n^{-\alpha} =0$ if and
only if its tail process tends to zero, \ie\ (\ref{eq:Y_to_zero}) holds and that in that case the
candidate extremal index $\vartheta$ is the true extremal index, \ie\
\begin{align*}
  \lim_{n\to\infty} \pr\left(\max_{1\leq i \leq n} \zeta_i \leq n^{1/\alpha} x\right) = \rme^{-\vartheta x^{-\alpha}} \; .
\end{align*}

\cite{janssen:2017} proves that given a non negative sequence $\boldsymbol\Theta$ which satisfies
$\Theta_0=1$ and the time change formula, there exists a max-stable process $\bzeta$ whose spectral
tail process is~$\vbTheta$. We provide a proof of this fact based on the tail measure when the tail
process tends to zero.

Let $\vbTheta$ be a non negative sequence wich satisfies the time change
formula~(\ref{eq:time-shift-Theta}) and such that $\pr(\Theta_0=1)=1$ and
$\lim_{|j|\to\infty} \Theta_j=0$. Define $\vartheta=\pr(I(\Theta)=0)$ and the measure $\nu$ on
$\Rset^\Zset$ by
\begin{align}
  \label{eq:def-nu}
  \nu(H) = \sum_{j\in\Zset} \int_0^\infty \esp[H(r\shift^j\vbTheta)\ind{I(\vbTheta)=0}] \alpha r^{-\alpha-1} \rmd r   \; .
\end{align}
Let $\sum_{i\geq1} \delta_{W^{(i)}}$ be a Poisson point process on $[0,\infty)^\Zset$ with mean
measure $\nu$. Define the max-stable process $\bzeta$ by
\begin{align} 
  \label{eq:def-zeta}
  \zeta_j = \bigvee_{i\geq1} W_j^{(i)} \; , \ j\in\Zset \; .
\end{align}
Let $Y_0$ be a Pareto random variable independent of $\vbTheta$ and define
$\vbY=Y_0\vbTheta$. Let $\vbQ$ be as in
\Cref{def:def-sequenceQ}. The following result proves the existence of a max-stable process with a
given spectral tail process and provides an M3 representation for it.  For a review of the M3
  representation of max-stable processes, see \cite{dombry:kabluchko:2016}.

\begin{theorem}
  \label{prop:given-theta}
  The measure $\nu$ given by~(\ref{eq:def-nu}) is $\sigma$-finite, $\nu(\{\bszero\})=0$, $\nu$ is
  homogeneous and shift invariant.  The max-stable process $\bzeta$ defined by (\ref{eq:def-zeta})
  is stationary, has tail measure $\nu$, spectral tail process $\vbTheta$, extremal index
  $\vartheta>0$ and it admits the M3 representation
  \begin{align}
      \{\zeta_j,j\in\Zset\} \stackrel{d}= \{\bigvee_{i\geq1} P_i Q^{(i)}_{j-T_i}  , \ j\in\Zset\} \; ,  \label{eq:M3-Q}
  \end{align}
  where $\sum_{i=1}^\infty \delta_{P_i}$ is a Poisson point process on $(0,\infty)$ with mean
  measure $\alpha x^{-\alpha-1} \rmd x$, $\vbQ^{(i)}$, $i\geq1$ are \iid\ copies of the sequence
  $\vbQ$  and are independent of the previous point process
  and $\sum_{i=1}^\infty \delta_{T_i}$ is a point process on $\Zset$ with mean measure $\vartheta$
  times the counting measure.
\end{theorem}

\begin{remark}
  It seems that the link between the sequence $\vbQ$ and the M3 representation was not known.
\end{remark}

\begin{proof}
  The fact that $\nu(\{\bszero\})=0$, the homogeneity and shift-invariance of $\nu$ are
  straightforward consequences of the definition. We prove that $\nu$ is $\sigma$-finite. In view of
  homogeneity and shift-invariance, it suffices to prove that $\nu(\{y_0>1\})<\infty$. For a
  measurable $A$, we have
  \begin{align*}
    \nu(A\cap\{y_0>1\}) 
    & =  \sum_{j\in\Zset} \int_0^\infty \pr(r\shift^j\vbTheta\in A, r(\shift^j\vbTheta)_0 >1,I(\vbTheta)=0) \alpha r^{-\alpha-1} \rmd r \\
    & =  \sum_{j\in\Zset} \esp \left[ \int_0^\infty \ind{r\shift^j\vbTheta\in A}\ind{r(\shift^j\vbTheta)_0 >1}\ind{I(\shift^j\vbTheta)=j} 
      \alpha r^{-\alpha-1} \rmd r \right] \; .
  \end{align*}
  The function
  $\vby \to \int_0^\infty \ind{r\vby\in A}\ind{ry_{0}>1}\ind{I(\vby)=j} \alpha
  r^{-\alpha-1} \rmd r $
  is $\alpha$-homogeneous and is equal to zero if $y_{0}=0$. Thus, applying the time change
  formula~(\ref{eq:time-shift-homo-alpha}) yields
  \begin{align*}
    \nu(A\cap\{y_0>1\}) 
    & =  \sum_{j\in\Zset} \esp \left[ \int_0^\infty \ind{r\vbTheta\in A}\ind{r\Theta_{0}>1}\ind{I(\vbTheta)=j} 
      \alpha r^{-\alpha-1} \rmd r \right] \\
    & = \esp \left[ \int_0^\infty \ind{r\vbTheta\in A}\ind{r\Theta_{0}>1} \alpha r^{-\alpha-1} \rmd r \right] \\
    & = \esp \left[ \int_1^\infty \ind{r\vbTheta\in A} \alpha r^{-\alpha-1} \rmd r \right] = \pr(\vbY \in A)  \; .
  \end{align*}
  Taking $A=\Rset^\Zset$ yields $\nu(\{y_0>1\})=1$.

  By \Cref{theo:tail-measure-infargmax}, to prove that $\nu$ is the tail measure of $\bzeta$, it
  suffices to prove that $\vbY$ is its tail process. By definition of $\bzeta$, we have, for
  $\vby\in[0,\infty]^\Zset$ with finitely many finite coordinates, as $u\to\infty$, 
    \begin{align*} \pr(\bzeta \in u [\bszero,\vby] \mid \zeta_0>u)
      & = \frac{\rme^{-u^{-\alpha}\nu([0,\vby]^c)}- \rme^{-u^{-\alpha}\nu(\{y_0>1\}\cup[0,\vby]^c)}}{1-\rme^{-u^{-\alpha}}} \\
      & \to  \nu(\{y_0>1\}\cup[0,\vby]^c) - \nu([0,\vby]^c) \\
      & =\nu(\{y_0>1\}\cap[0,\vby]) = \pr(\vbY\in[0,\vby]) \; .
  \end{align*}
  This proves that $\vbY$ is the tail process of $\bzeta$ and that $\nu$ is the tail measure of
  $\bzeta$.

  To prove (\ref{eq:M3-Q}), it suffices to note that for $\vbx\in(0,\infty]^\Zset$ with only
  finitely many finite coordinates, denoting $\xi$ the process in the right hand side
  of~(\ref{eq:M3-Q}), we have
  \begin{align*}
    - \log \pr(\xi_j\leq x_j,j\in\Zset) 
    & = \vartheta \sum_{i\in\Zset} 
      \int_0^\infty \pr\left(r\bigvee_{j\in\Zset} \frac{Q_{j-i}}{x_j} \leq 1\right) \alpha r^{-\alpha-1} \rmd r 
     =  \nu(\{\vby,y_j\leq x_j,j\in\Zset\}) \; .
  \end{align*}
\end{proof}

\paragraph{Acknowledgement} \Cref{sec:maxstableprocesses} owes a lot to Enkelejd Hashorva who
brought the references \cite{hashorva:2016,debicki:hashorva:2017} to our attention as well as the
formula~(\ref{eq:candidate-hashorva}). The research of the first author is supported in part by Croatian
Science Foundation under the project 3526. The research of the second author is partially supported by
LABEX MME-DII.

\appendix

\section*{Appendix}

\section{Proof of the equivalence between (\ref{eq:time-shift-Y})
  and~(\ref{eq:time-shift-Theta})}

Assume first that (\ref{eq:time-shift-Y}) holds. It suffices to prove (\ref{eq:time-shift-Theta})
for a non negative measurable functional~$H$, homogeneous with degree~0. Applying~(\ref{eq:polar})
and  the monotone convergence theorem, we obtain
  \begin{align*}
    \lim_{t\to0} \esp [ H(B^k\vbY) \ind{|\vbY_{-k}|>t}] 
    & = \lim_{t\to 0}\esp [H(B^k\vbTheta) \ind{|\vbY_{-k}|>t}]  \\
    & = \esp [H(B^k\vbTheta) \ind{|\vbTheta_{-k}|>0}]  \; .
  \end{align*}
  On the other hand, applying again~(\ref{eq:polar}) and the monotone convergence theorem yields
  \begin{align*}
    \lim_{t\to0}    t^{-\alpha} \esp [ H(t\vbY) \ind{|\vbY_{k}|>1/t}]
    & = \lim_{t\to0} t^{-\alpha} \esp \left[ H(\vbTheta) \int_1^\infty \ind{r\norm{\vbTheta_{k}}>1/t} \alpha r^{-\alpha-1}\rmd r \right]  \\
    & = \lim_{t\to0}  \esp \left[ H(\vbTheta) \left(\norm{\vbTheta_{k}} \wedge 1/t\right)^\alpha \right] \\
    & = \esp \left[ H(\vbTheta) \norm{\vbTheta_{k}} ^\alpha \right] \; .
  \end{align*}
  Since we started from quantities which are equal by (\ref{eq:time-shift-Y}), this
  proves~(\ref{eq:time-shift-Theta}) for a $0$-homogeneous functional.

  Conversely, assume that (\ref{eq:time-shift-Theta}) holds and let $H$ be a non negative measurable
  functional and $t>0$. Then by~(\ref{eq:polaire-tail-measure-0})
  \begin{align*}
    \esp & [ H(B^k\vbY)  \ind{|\vbY_{-k}|>t}]  \\
         & = \int_1^\infty \esp [ H(rB^k\vbTheta) \ind{r\norm{\vbTheta_{-k}}>t}] \alpha r^{-\alpha-1} \rmd r \\
         & = \int_1^\infty \esp [ H(rB^k\vbTheta) \ind{r\norm{\left(B^k\vbTheta\right)_{0}}>t} 
           \ind{\norm{\vbTheta_{-k}}\ne0}] \alpha r^{-\alpha-1} \rmd r \\
         & = \int_1^\infty \esp [ H(r\norm{\vbTheta_k}^{-1}\vbTheta) \ind{r\norm{\vbTheta_k}^{-1}\norm{\vbTheta_{0}}>t}
           \norm{\vbTheta_k}^\alpha] \alpha r^{-\alpha-1} \rmd r \\
         & = t^{-\alpha} \esp \left[ \int_1^\infty  H(tu\vbTheta) \ind{u\norm{\vbTheta_k}>1/t}   \alpha u^{-\alpha-1} \rmd u \right]
           = t^{-\alpha} \esp \left[H(t\vbY) \ind{\norm{\vbY_k}>1/t} \right]   \; ,
  \end{align*}
  where the last line was obtained by the change of variable $u\norm{\vbTheta_k}t=r$. Thus
  (\ref{eq:time-shift-Y}) holds.

\section{Proof of~(\ref{eq:borne-diff-slogs})}
\label{sec:properties-slogs}

Let $g$ be defined on $\Rset$ by $g(x) = x\log(|x|)$ with the convention $0\log 0 =0$.  Then
$|g(x)|\leq 1$ for $x\in[-1,1]$ and if $|x|\vee|y| \geq 1$ and $|x-y|\leq1$, which implies that $x$
and $y$ are of the same sign and $||x|-|y||=|x-y|\leq1$, we have
\begin{align*}
   |g(x)-g(y)| = \int_{|x|\wedge|y|y}^{|x|\vee|y|} \{1+\log(s)\}\rmd s \leq ||x|-|y||(\log(|x|\vee|y|)+1)\leq \log(|x|\vee |y|)+1 \; . 
\end{align*}
For $\vbx,\vby\in\ell_1$ such that $|S(\vbx)-S(\vby)|\leq1$, this yields
\begin{align*}
  |g(S(\vbx))& -g(S(\vby))|  \\
  & \leq 2 \cdot \ind{|S(\vbx)|\vee |S(\vby)| \leq1} 
   + (\log(\norm[1]{\vbx}\vee\norm[1]{\vby}) +1)\ind{|S(\vbx)|\vee |S(\vby)| \geq 1} \\
  &  \leq 2 + \log_+(\norm[1]{\vbx}\vee\norm[1]{\vby})  \; .
\end{align*}

\end{document}